\documentclass[11pt]{paper}

\usepackage{cite}
\usepackage{comment}
\usepackage{pslatex}
\usepackage{amsmath,amsthm}
\usepackage{amssymb}            
\usepackage{amsfonts}           
\usepackage{graphicx}
\usepackage{fancyhdr}

\usepackage{graphicx}
\usepackage{tikz}
\usetikzlibrary{positioning, calc}

\usepackage[utf8]{inputenc}

\usepackage[english]{babel}

\usepackage{amssymb,amsmath}
\usepackage{amsmath,subfigure,bm}

\usepackage{color}

\usepackage{verbatim}

\usepackage[numbers, square, comma, sort&compress]{natbib}

\usepackage{hyperref}
\newcommand{\pder}[2][]{\frac{\partial #1}{\partial #2}}
\def\reals{\mathbb R}
\def\p{\partial}
\def\WS{{\mathcal W}}
\def\deg{{M_{\text{D}}}}

\def\mleg{{M_{\text{L}}}}
\def\beq{\begin{equation}}
\def\eeq{\end{equation}}
\def\meq{{M_{\text{E}}}}
\def\bxi{\boldsymbol{\xi}}

\newcommand\avg[1]{\overline{#1}}

\newcommand\R{\mathbb{R}}

\newcommand{\Tm}{\mathcal{T}}

\newcommand{\pd}[2]{\frac{\partial #1}{\partial #2}}            
\newcommand{\pdn}[3]{{\frac{\partial^{#1}#2}{\partial#3^{#1}}}} 

\newcommand{\dt}{\Delta t}

\newcommand{\BigOh}{\mathcal{O}}

\newcommand{\eps}{\varepsilon}              

\newcommand{\E}{\mathcal{E}}

\newtheorem{rmk}{Remark}
\newtheorem{lem}{Lemma}

\newcommand{\hf}{\frac{1}{2}}

\newcommand{\En}{\mathcal{E}}

\providecommand{\norm}[1]{\lVert#1\rVert}

\usepackage{authblk}

\title{\center Positivity-preserving discontinuous Galerkin methods with Lax-Wendroff time discretizations}
\author[1]{\normalsize Scott A. Moe}
\author[2]{James A. Rossmanith}
\author[3]{David C. Seal}
\affil[1]{\it University of Washington, Department of Applied Mathematics, Seattle, WA 98195, USA ({\tt smoe@uw.edu})}
\affil[2]{\it Iowa State University,
      Department of Mathematics,
       396 Carver Hall, Ames, IA 50011, USA ({\tt rossmani@iastate.edu})}
\affil[3]{\it U.S. Naval Academy,
    Department of Mathematics,
	121 Blake Road,
	Annapolis, MD 21402, USA ({\tt seal@usna.edu})}

\begin{document}




\maketitle

\begin{abstract}

This work introduces a single-stage, single-step method for the compressible
Euler equations that is provably positivity-preserving and can be applied
on both Cartesian and unstructured meshes.  This method
is the first case of a 
single-stage, single-step method that is simultaneously
high-order, positivity-preserving, and operates on unstructured meshes.
Time-stepping is accomplished via the Lax-Wendroff approach, which
is also sometimes called the Cauchy-Kovalevskaya procedure,
where temporal derivatives in a Taylor series in time are exchanged for spatial derivatives.
The Lax-Wendroff discontinuous Galerkin (LxW-DG) method developed in this work
is formulated so that it looks like a forward Euler update but with a high-order 
time-extrapolated flux. 
In particular, the numerical flux used in this work is a linear combination of a low-order positivity-preserving contribution and a high-order component that can be damped to enforce positivity of the cell averages for the density and
pressure for each time step. In addition to this flux limiter, 
a moment limiter is
applied that forces positivity of the solution at finitely many quadrature
points within each cell. The combination of the flux limiter and the
moment limiter guarantees positivity of the cell averages from one time-step
to the next. Finally, a simple shock capturing limiter that uses
the same basic technology as the moment limiter is introduced
in order to obtain non-oscillatory results. The resulting scheme
can be extended to arbitrary order without increasing the size of
the effective stencil.  We present numerical results in one and two space dimensions 
that demonstrate the robustness of the proposed scheme.


\end{abstract}





\section{Introduction}
\label{sec:introduction}

\subsection{\it Governing equations}

The purpose of this work is to develop a positivity-preserving version of 
the Lax-Wendroff discontinuous Galerkin method for
the compressible Euler equations on unstructured meshes.  The compressible Euler equations 
form a system of hyperbolic conservation law that can be written as follows:
\begin{equation}
\label{eqn:euler-eqns}
\left( 
\begin{array}{c}
    \rho \\ \rho \vec{u} \\ \E
    \end{array}
    \right)_{,t}
    + 
    \nabla_{\bf x} \cdot
    \left( 
    \begin{array}{c}
       \rho \vec{u} \\ \rho \norm{ \vec{u} }^2 + p \\ ( \E + p ) \vec{u}
    \end{array}
    \right)
    = 0.
\end{equation}
%
The {\it conserved} variables are the mass density, $\rho$, the momentum density, $\vec{M}
= \rho \vec{u}$, and the energy density, $\E$; the {\it primitive} variables are
the mass density, $\rho$, the fluid velocity, $\vec{u}$, and the pressure, $p$.
The energy $\E$ is related to the primitive variables through 
the equation of state, 
\begin{equation}
\label{eqn:eos}
\E = \frac{p}{ \gamma-1 } + \frac{1}{2}\rho \norm{ \vec{u} }^2,
\end{equation}
where the constant $\gamma$ is the ratio of specific heats (aka, the {\it gas constant}).  



The compressible Euler equations are an important mathematical model
in the study of gases and plasma. Attempts at numerically solving
the these equations has led to a plethora of important historical advances 
in the development of numerical analysis and scientific computing (see e.g.,
\cite{Godunov,Lax,CourantDifferences,VNeumann,LaxWendroff}).


\subsection{\it Discontinuous Galerkin spatial discretization}
\label{sec:dg_spatial}
The focus of this work is on high-order discontinuous Galerkin (DG) methods,
which were originally developed for general hyperbolic conservation laws
by Cockburn, Shu, et al. in series of
papers \cite{TVBRKDG2,TVBRKDG3,TVBRKDG4,TVBRKDG5,CoKarn}. 
The purpose of this section is to set the notation used throughout 
the paper and to briefly describe the DG spatial discretization.

Let $\Omega \subset \reals^d$ be a polygonal domain with boundary $\partial \Omega$. 
The domain $\Omega$ is discretized via a finite set of non-overlapping elements, $\Tm_i$, such
that  $\Omega = \cup_{i=1}^N \Tm_i$. 
Let ${P}^{\, \deg}\left(\reals^d\right)$ denote the set of polynomials from $\reals^d$ to $\reals$ 
with maximal polynomial degree $\deg$.
Let $\WS^{h}$ denote the {\it broken} finite element space on the mesh:
\begin{equation}
\label{eqn:broken_space}
   \WS^h := \left\{ w^h \in \left[ L^{\infty}(\Omega) \right]^{\meq}: \,
    w^h \bigl|_{\Tm_i} \in \left[ {P}^{\, \deg} \right]^{\meq}, \, \forall \Tm_i \in \Tm^h \right\},
\end{equation}
where $h$ is the mesh spacing.
The above expression means that $w^h \in \WS^{h}$ has
$\meq$ components, each of which when restricted to some element $\Tm_i$
is a polynomial of degree at most $\deg$ and no continuity is assumed
across element edges (or faces in 3D). 

The approximate solution on each element $\Tm_i$ at time $t=t^n$ is of the form
\beq
q^h(t^n, {\bf x}(\bxi)) \Bigl|_{\Tm_i} = \sum_{\ell=1}^{\mleg(\deg)} \, Q^{(\ell) n}_i  \,
\varphi^{(\ell)}\left( \boldsymbol{\xi} \right),
\eeq
where $\mleg$ is the number of Legendre polynomials and
$\varphi^{(\ell)}\left( \bxi \right): \reals^d \mapsto \reals$ are the Legendre polynomials defined
on the reference element ${\mathcal T}_0$ in terms of the reference coordinates $\bxi \in \Tm_0$.
The Legendre polynomials are orthonormal with respect to the following inner product:
\beq
\frac{1}{|\Tm_0|} \int_{\Tm_0} \varphi^{(k)}(\bxi) \, \varphi^{(\ell)}(\bxi) \, d\bxi = 
\begin{cases}
1 & \, \text{if} \quad k=\ell, \\
0 & \, \text{if} \quad k \ne \ell.
\end{cases}
\eeq
We note that independent of $h$, $d$, $\deg$,  and the type of element, the lowest order Legendre polynomial is always $\varphi^{(1)} \equiv 1$. This makes the first Legendre coefficient the
cell average:
\beq
    Q^{(1) n}_i = \frac{1}{|\Tm_0|} \int_{\Tm_0}
    q^h(t^n, {\bf x}(\bxi)) \Bigl|_{\Tm_i}  \, \varphi^{(1)}\left( \bxi \right) \,  d\bxi =: \avg{ q }^n_i.
\eeq

\subsection{\it Time stepping}
The most common approach for time-advancing DG spatial discretizations is via
explicit Runge-Kutta time-stepping; the resulting combination of time and space discretization is often
referred to as the ``RK-DG'' method \cite{TVBRKDG4}.  The primary advantage for this choice of time stepping
 is that explicit RK methods are easy to implement, they can be constructed to be
low-storage, and a subclass of these methods have the so-called strong
stability preserving (SSP) property \cite{GoShuTa01},
which is important for defining a scheme that is provably positivity-preserving.  
However, there are no explicit Runge-Kutta methods that are SSP for orders greater
than four \cite{Kraa1991,RuuthSpiteri2002}.

The main difficulty with Runge-Kutta methods is that they typically require many stages; and therefore, many communications are needed per time step.
One direct consequence of the communication required at each RK stage
is that is difficult to combine RK-DG with locally adaptive mesh refinement strategies 
that simultaneously refine in both space and time.
 
The key piece of technology required in locally adaptive DG schemes is
local time-stepping (see e.g., Dumbser et al. \cite{dumbser2007arbitrary}).
Local time-stepping is easier to accomplish with a single-stage, single-step (Lax-Wendroff) method than with a multi-stage Runge-Kutta scheme. 
For these reasons there is interest from discontinuous Galerkin theorists and
practitioners in developing single-step time-stepping techniques for DG
(see e.g.,
\cite{proceedings:TitarevToro02,QiuDumbserShu05,dumbser2005ader,DuMu06,TaDuBaDiMu07,DuBaDiToMuDi08,Gassner11,DuZaHiBa13}), 
as well as hybrid multistage multiderivative alternatives \cite{SeGuCh14}.

In this work, we construct a numerical scheme that uses a Lax-Wendroff
time discretization that is coupled with the discontinuous Galerkin spatial discretization. In subsequent discussions in this paper we demonstrate 
the advantages of switching to single-stage and single-stage 
time-stepping in regards to enforcing positivity on arbitrary meshes.

\subsection{\it Positivity preservation}

In simulations involving strong shocks, high-order schemes (i.e. more than first-order) for the compressible Euler equations generally create
nonphysical undershoots (below zero) in the density and/or pressure.  These undershoots 
typically cause catastrophic numerical instabilities due to a loss of hyperbolicity.
Moreover,  for many applications these positivity violations exist even when the equations are coupled 
with well-understood total variation diminishing (TVD) or total variation bounded (TVB) limiters.
The chief goal of the limiting scheme developed in this work
is to address positivity violations in density and pressure, and in particular, to accomplish this task with a high-order scheme.

In the DG literature, the most widely used strategy to maintain positivity was developed by Zhang and Shu
in a series of influential papers \cite{article:ZhangShu10rectmesh,ZhangShu11,ZhangShuTrimesh}.
The basic strategy of Zhang and Shu for a positivity-preserving RK-DG method can
be summarized as follows:
\begin{description}
\item[{\bf Step 0.}] Write the current solution in the form
\begin{equation}
	q^h \bigl|_{\Tm_i} = \avg{q}_i + \theta \left( q^h \bigl|_{\Tm_i} - 
	\avg{q}_i \right),
\end{equation}
where $\theta$ is yet-to-be-determined. $\theta = 1$ represents the unlimited solution.
\item[{\bf Step 1.}] 
Find the largest value of $\theta$, where $0 \le \theta \le 1$,
such that $q^h \bigl|_{\Tm_i}$ satisfies the appropriate positivity conditions
at some appropriately chosen quadrature points and limit the solution.
\item[{\bf Step 2.}] Find the largest stable time-step that
guarantees that with a forward Euler time that the cell average of the new
solution remains positive.
\item[{\bf Step 3.}] Rely on the fact that strong stability-preserving Runge-Kutta methods are convex combinations of forward Euler time steps;
and therefore, the full method preserves the positivity of cell averages
(under some slightly modified maximum allowable time-step).
\end{description}

For a Lax-Wendroff time discretization, numerical results indicate that the
  limiting found in Step 1 is insufficient to retain positivity of the
  solution, even for simple 1D advection.  Therefore, the strategy we pursue
  in this work will still contain an equivalent Step 1; however, in place of
  Step 2 and Step 3 above, we will make use of a parameterized flux, sometimes
  also called a flux
  corrected transport (FCT) scheme, to maintain positive cell averages
  after taking a single time step. In doing so, we avoid introducing
  additional time step restrictions that often appear (e.g., in Step 2. above)
  when constructing a positivity-preserving scheme based on Runge-Kutta time
  stepping.

This idea of computing modified fluxes by combining a stable low-order flux
with a less robust high-order flux is relatively old, and perhaps originates
with Harten and Zwas and their \emph{self adjusting hybrid scheme}
\cite{harten1972self}.  The basic idea is the foundation of the related
\emph{flux corrected transport} (FCT) schemes of Boris, Book and collaborators \cite{boris1973flux,book1975flux,boris1976flux,book1981finite},
where fluxes are adjusted in order to guarantee that average values of  the unknown are constrained to lie within 
locally defined upper and lower bounds.
This family of methods is used in an extensive variety of applications, ranging from seismology to meteorology \cite{ZalesakStruct,kuzmin05,zheng2006non,ullrich2014flux}.
A thorough analysis of some of the early methods is conducted in \cite{SOD}. Identical to modern maximum principle preserving (MPP) schemes, FCT can be formulated as a global optimization problem
where a ``worst case'' scenario assumed in order to decouple the previously coupled degrees
of freedom \cite{FCTPavel}. 
Here we do not attempt to use FCT to enforce any sort of local bounds (in the sense of developing a shock-capturing limiter), 
instead we leverage these techniques in order to retain positivity of the density and pressure
associated to $q^h(t^n, \vec{x})$; such approaches have recently
received renewed interest in the context of weighted essentially non-oscillatory (WENO) methods 
\cite{xu2013,liang2014parametrized,tang14,christlieb2015high,seal2014explicit,ChFeSeTa2015}.

To summarize, our limiting scheme draws on ideas from the two aforementioned families of techniques that are well established in the literature. First,
we start with the now well known (high-order) pointwise limiting
developed for discontinuous Galerkin methods \cite{article:ZhangShu10rectmesh,ZhangShu11,ZhangShuTrimesh},
and second, we couple this with the very large family of flux limiters \cite{christlieb2015high,seal2014explicit,ChFeSeTa2015}.
(developed primarily for finite-difference (FD) and finite-volume (FV)
schemes). 
\subsection{\it An outline of the proposed positivity-preserving method}


The compressible Euler equations \eqref{eqn:euler-eqns}
can be written compactly as
\begin{equation}
\label{eqn:conslaw}
    q_{,t} + {\nabla} \cdot {\bf F}(q) = 0,
\quad \text{in} \, \, \, \Omega \subset \reals^d,
\end{equation}
where the {\it conserved variables} are $q = (\rho, {\bf M}, \En)$ 
and the {\it flux function} is
\begin{equation}
\vec{F} \cdot \vec{n} = \begin{pmatrix}
 \vec{M} \cdot \vec{n} \\
 \left(\vec{M} \cdot \vec{n}\right) \vec{u} + p \vec{n} \\ 
\vec{u} \cdot \vec{n} \left( \En + p \right)
\end{pmatrix},
\end{equation}
where $\vec{M} = \rho \vec{u}$ and $\En = \frac{p}{\gamma-1} + \frac{1}{2} \rho \| \vec{u} \|^2$.


The basic positivity limiting strategy proposed in this work is
summarized below. Some important details are omitted here, but we
elaborate on these details in subsequent sections.

\begin{description}
\item[{\bf Step 0.}] On each element we write the solution as
\begin{equation}
	q^h(t^n, \vec{x}(\vec{\xi}))\Bigl|_{\Tm_i} := 
	\avg{q}^n_i + 
	\theta \sum_{k=2}^{\mleg(\deg)} Q^{(k)}_{i}(t) \varphi^{(k)}(\vec{x}),
\end{equation}
where $\theta$ is yet-to-be-determined. $\theta=1$ represents the 
{\it unlimited} solution.

\item[{\bf Step 1.}] Assume that this solution is positive in the
mean.  That is, we assume for all $i$ that $\avg{\rho}_i > 0$ and 
\begin{equation}
\label{eqn:ave_press}
	\avg{p}_i := (\gamma - 1) \avg{\E}_i 
    - \frac{1}{2} \frac{ \| {\bf \avg{M}}_i \|^2 }{ \avg{\rho}_i } > 0.
\end{equation}
A consequence of these assumptions is that $\avg{\E}_i > 0$.

\item[{\bf Step 2.}] Find the largest value of $\theta$, where $0 \le \theta \le 1$,
such that the density and pressure are positive at some suitably defined quadrature
points.
%
This step is elaborated upon in \S\ref{subsec:mpp-limiter}.

\item[{\bf Step 3.}] Construct time-averaged fluxes through the Lax-Wendroff procedure.  That is, we start with the exact definition of the time-average flux:
\begin{equation}
\label{eqn:picard1d-b}
    \avg{\bf F}^n(\vec{x}) := \frac{1}{\dt} \int_{t^n}^{t^{n+1}} {\bf F}( q(t,\vec{x} ) )\, dt
    = \frac{1}{\dt} \int_{0}^{\Delta t} {\bf F}( q(t^n+s,\vec{x} ) )\, ds
\end{equation}
and approximate this via a Taylor series expansion around $s=0$:
\begin{equation}\label{eqn:1D_system.TI-F}
  {\bf F}_T^n(\vec{x}) := {\bf F}( q(t^n,\vec{x}) ) 
    + \frac{\dt}{2!}   \frac{d {\bf F}}{dt} ( q(t^n,\vec{x}) ) 
    + \frac{\dt^2}{3!} \frac{d^2 {\bf F}}{dt^2}( q(t^n,\vec{x}) ) 
    = \avg{\bf F}^n(\vec{x}) + \BigOh(\dt^3).
\end{equation}
All time derivatives in this expression are replaced by spatial derivatives using
the chain rule and the governing PDE \eqref{eqn:conslaw}.
The approximate time-averaged flux \eqref{eqn:1D_system.TI-F} is first evaluated
at some appropriately chosen set of quadrature points -- in fact, the same
quadrature points as used in {Step 2} -- and then, using appropriate quadrature weights, summed together to define a high-order flux, at both interior and boundary quadrature points.  Thanks to {Step 2}, all quantities of
interest used to construct this expansion are positive at each quadrature
point.

%
\item[{\bf Step 4.}] Time step the solution so that cell averages are guaranteed to be positive.
That is, we update the \emph{cell averages} via a formula of the form
\begin{equation}
	Q^{(1) n+1}_i = Q^{(1) n}_i - \frac{\dt}{|\Tm_i|}
	\sum_{e \in {\Tm_i}}
	{\vec{F}}^{h*}_{e} \cdot \vec{n}_e,
\end{equation}
where $\vec{n}_e$ is an outward-pointing (relative to $\Tm_i$) 
normal vector to edge $e$
with the property that $\| \vec{n}_e \|$ is the length of edge $e \in {\Tm_i}$,
and the numerical flux on edge, ${\vec{F}}^{h*}_e$, is 
a convex combination of a high-order flux, ${\mathcal F}^H_e$, and a low-order flux ${\mathcal F}^L_e$:
\begin{equation}
	{\vec{F}}^{h*}_{e } := \theta {\mathcal F}_{e }^{\text{H}} + \left( 1 - \theta \right) 
	{\mathcal F}_{e}^{\text{L}}.
\end{equation}
The low-order flux, ${\mathcal F}_e^{\text{L}}$, is based on the (approximate) solution to the Riemann problem defined by cell averages only,  
and the ``high-order'' flux, ${\mathcal F}_e^{\text{H}}$, is constructed after integrating via Gaussian quadrature the (approximate) Riemann solutions 
 at quadrature points along the edge $e \in \Tm_i$:
\begin{equation}
	{\mathcal F}^{\text{H}}_{e } = \frac{1}{2} \sum_{k=1}^{{M}_{\text Q}} \omega_k 
		{\mathcal F}_{ek}^{\text{H}},
\end{equation}
where $\omega_k$ are the Gaussian quadrature weights for quadrature with 
${M}_{\text Q}$ points and  ${\mathcal F}_{ek}^{\text{H}}$ are the
numerical fluxes at each of the ${M}_{\text Q}$ quadrature points.
Note that this sum has only a single
summand in the one-dimensional case.
The selection of $\theta$ is described in more detail
in \S\ref{subsec:mpp-limiter}.
This step guarantees that the solution retains positivity (in the mean) for a single time step.

\item[{\bf Step 5.}] Apply a shock-capturing limiter.  The positivity-preserving limiter is designed to preserve positivity of the solution, 
but it fails at reducing spurious oscillations, and therefore a shock-capturing limiter needs to be added.  There are many choices of limiters available; 
we use the limiter recently developed in
\cite{MoeRossSe15} because of its ability to retain genuine high-order accuracy, and its ability to push the polynomial order to arbitrary degree without modifying the overall scheme.

\item[{\bf Step 6.}] Repeat all of these steps to update the solution for the
next time step.

\end{description}


Each step of this process is elaborated upon throughout the remainder of this
paper. The end result is that our method is the first scheme to simultaneously obtain
all of the following properties:
\begin{itemize}
\item {\bf High-order accuracy.}  The proposed method is third-order in space and time, and can
be extended to arbitrary order.
\item {\bf Positivity-preserving.}  The proposed limiter is provably 
positivity-preserving for the density and pressure, at a finite set of point values, for the entire simulation.
\item {\bf Single-stage, single-step.}  We use a Lax-Wendroff discretization
for time stepping the PDE, and therefore we only need one communication per 
time step.
\item {\bf Unstructured meshes.}  Because we use the discontinuous Galerkin method
for our spatial discretization and all of our limiters are sufficiently
local, we are able to run simulations with DG-FEM on both Cartesian and unstructured meshes.  
\item {\bf No additional time-step restrictions.} Because we do not rely on a SSP Runge-Kutta scheme, we do not
have to introduce additional time-step restrictions to retain positivity of
the solution.  This differentiates us from popular positivity-preserving
limiters based on RK time discretizations \cite{ZhangShu11}.
\end{itemize}

\subsection{\it Structure of the paper}

The remainder of this paper has the following structure.
The Lax-Wendroff DG (LxW-DG) method is described in \S\ref{sec:LWDG},
where we view the scheme as a method of modified fluxes.
The positivity-preserving limiter is described in \S\ref{sec:positivity},
where the discussion of the limiter is broken up into two parts:
(1) the moment limiter (\S\ref{sec:zh-limiter}) and (2) 
the parameterized flux limiter (\S\ref{subsec:mpp-limiter}).
In \S\ref{sec:numerical-results} we present numerical results on several
test cases in 1D, 2D Cartesian, and 2D unstructured meshes.
Finally we close with conclusions and a discussion of future work in
\S\ref{sec:conclusions}.

\section{The Lax-Wendroff discontinuous Galerkin scheme}
\label{sec:LWDG}
\subsection{\it The base scheme: A method of modified fluxes}

The Lax-Wendroff discontinuous Galerkin (LxW-DG) method \cite{QiuDumbserShu05} serves as the base
scheme for the method developed in this work.
It is the result of an application of the Cauchy-Kovalevskaya
procedure to hyperbolic PDE: we start with a Taylor series in time,
then we replace all time derivatives with spatial derivatives via the PDE.
Finally, a Galerkin projection
discretizes the overall scheme, where a single spatial derivative is reserved
for the fluxes in order to perform integration-by-parts. 

We review the Lax-Wendroff DG scheme for the case of a general
nonlinear conservation law that is autonomous in space and time
in multiple dimensions \cite{QiuDumbserShu05}.  The current presentation
illustrates the fact that Lax-Wendroff schemes can be viewed as a method of
modified fluxes, wherein higher-order information about the PDE is directly 
incorporated by simply redefining the fluxes that would typically be used in
an ``Euler step.''

We consider a generic conservation law of the form
\begin{equation}\label{eq:fluxform}
    q_{,t} + \nabla \cdot {\bf F}(q) = 0,
\end{equation}
where the matrix $\pd{{\bf F}}{q} \cdot \hat{ \bf n }$ is diagonalizable for
every unit length vector $\hat{ \bf n }$ and $q$ in the domain of interest.
Formal integration of  \eqref{eq:fluxform} over an interval $[t^n, t^{n+1}]$ results in
an exact update through
%
\begin{equation}
\label{eqn:exact-taylor}
    q(t+\Delta t, \vec{x}) = q(t, \vec{x})-\Delta t \, \nabla \cdot \avg{ \bf F}( q(t,\vec{x}) ),
\end{equation}
where the \emph{time-averaged} flux \cite{SeGuCh15} is defined as
\begin{equation}
    \avg{{\bf F}}( q(t,\vec{x}) ) := \frac{1}{\Delta t}\int_{t^n}^{t^{n+\Delta t}} {\bf F}( q(t,\vec{x}) ) \, dt.
\end{equation}
Moreover, a Taylor expansion of ${\bf F}$ and a change of variables yields
\begin{equation}
\label{eq:timeavflux}
    \begin{aligned}
     \avg{{\bf F}}(q)
     &= \frac{1}{\Delta t}\int_{0}^{\Delta t}
     \left({\bf F}(q^n)+\tau \, {\bf F}(q^n)_{,t}+\frac{1}{2} \tau^2 \, {\bf F}(q^n)_{,t,t}+\cdots\right)d\tau \\
     &= {\bf F}(q^n)+\frac{1}{2!} \Delta t \, {\bf F}(q^n)_{,t}+\frac{1}{3!} \Delta t^2 \,
     {\bf F}(q^n)_{,t,t}+\cdots,
    \end{aligned}
\end{equation}
which can be inserted into \eqref{eqn:exact-taylor}. 
In a numerical discretization of \eqref{eqn:exact-taylor}, the Taylor series
in \eqref{eq:timeavflux} is truncated after a finite number of terms.
%
\begin{rmk}
If $\avg{\bf F} \approx {\bf F}(q(t^n,\vec{x}))$, then \eqref{eqn:exact-taylor}
reduces to a forward Euler time discretization for hyperbolic
conservation law \eqref{eq:fluxform}. 
This fact will allow us to incorporate positivity-preserving
limiters into the Lax-Wendroff flux construction.
\end{rmk}
%
This observation allows us to incorporate the positivity-preserving
limiters that are presented in \S\ref{sec:zh-limiter} and \S\ref{subsec:mpp-limiter}, because we view the LxW-DG method as a method of modified fluxes.

\subsection{\it Construction of the time-averaged flux}

We now describe how to compute the temporal derivative terms:
\begin{equation}
 {\bf F}(q^n)_{,t}, \quad {\bf F}(q^n)_{,t,t}, \quad {\bf F}(q^n)_{,t,t,t}, \quad \dots
\end{equation}
that are required to define the time-averaged flux in  \eqref{eq:timeavflux}.
This discussion is applicable to high-order finite difference methods, finite volume 
methods (e.g., ADER), as well as 
discontinuous Galerkin finite element methods.

A single application of the chain rule to compute the time derivative of the flux function yields
\begin{equation}
\label{eqn:fluxderivs-a}
    \pd{\bf F}{t} = {\bf F}'(q) \cdot q_{,t} = -{\bf F}'(q) \cdot \left( \nabla \cdot {\bf F} \right),
\end{equation}
where the flux Jacobian is
\begin{equation}
    {\bf F}'(q)_{ij} := \pd{ {\bf F}_i }{ q_j }, \quad
    1 \leq i, j \leq M.
\end{equation}

The matrix-vector products in \eqref{eqn:fluxderivs-a} can be
compactly written using the Einstein summation convention (where repeated
indices are assumed to be summed over), which produces a
vector whose 
$i^{\text{th}}$-component is
\begin{equation}
    \pder[{\bf F}_i]{t} = \pder[{\bf F}_i ]{q_j} \pder[q_{j}]{t} = - \pder[{\bf F}_i ]{q_j}  \left( \nabla \cdot {\bf F} \right)_{j}.
\end{equation}
A second derivative of  \eqref{eqn:fluxderivs-a} yields
\begin{equation}
\label{eqn:fluxderivs-b}
        \pdn{2}{\bf F}{t} 
=   \pd{}{t} \left( -{\bf F}'(q) \nabla \cdot {\bf F} \right)
= {\bf F}''(q) \cdot \left( \nabla \cdot {\bf F}(q), \nabla \cdot {\bf F}(q) \right) 
+ {\bf F}'(q) \cdot \nabla \left( {\bf F}'(q) \left( \nabla \cdot {\bf F} \right) \right),
\end{equation}
where ${\bf F}''(q)$ is the Hessian with elements given by
\begin{equation}
	{\bf F}_{ijk} := \frac{ \partial^2 {\bf F}_i }{ \partial q_j \partial q_k }
= \left( \frac{ \partial^2 f_i }{ \partial q_j \partial q_k }, \frac{ \partial^2 g_i }{ \partial q_j \partial q_k } \right).
\end{equation}
Equations \eqref{eqn:fluxderivs-a} and \eqref{eqn:fluxderivs-b} are
generic formulae; the equalities are appropriate for any two-dimensional
hyperbolic system, and similar identities exist for three dimensions.
The first product in the right hand side of  \eqref{eqn:fluxderivs-b} is
understood as a Hessian-vector product.  Scripts that compute these derivatives, as well as the matrix, and Hessian vector products
that are necessary to implement a third-order Lax-Wendroff scheme for multidimensional Euler equations can be found in the open source software FINESS \cite{FINESS}.

Finally, these two time derivatives are sufficient to construct a third-order
accurate method by defining the time-averaged flux through
\begin{equation}
\label{eqn:taylor3-flux}
     {\bf F}^n_{T}(q) :=
     {\bf F}(q^n)+\frac{1}{2!} \Delta t \, {\bf F}(q^n)_{,t} +\frac{1}{3!} \Delta
     t^2 \,  {\bf F}(q^n)_{,t,t},
\end{equation}
and then updating the solution through
\begin{equation}
\label{eqn:taylor3}
    q^{n+1}(\vec{x}) = q^n(\vec{x}) - \Delta t \, \nabla \cdot { \bf F}^n_T
\end{equation}
in place of \eqref{eqn:exact-taylor}.

\subsection{\it Fully-discrete weak formulation}
\label{sec:fully-discrete}
The final step is to construct a fully discrete version of
 \eqref{eqn:taylor3}.  The LxW-DG scheme follows the following process
\cite{SeGuCh14,GuoQiuQiu14}:
\begin{description}

\item[{\bf Step 1.}] At each quadrature point evaluate the numerical flux, ${\bf F}(q^n)$, and then integrate this numerical flux
against basis functions to obtain a Galerkin expansion of ${\bf F}^h$ inside
each element.  

\item[{\bf Step 2.}] Using the Galerkin expansions of 
$q^h$ and ${\bf F}^h$, evaluate all required spatial derivatives to construct the
time expansion ${\bf F}^n_T$ in  \eqref{eqn:taylor3-flux} at each
quadrature point.

\item[{\bf Step 3.}] Multiply  \eqref{eqn:taylor3} by a test function $\varphi^{(\ell)}$,
integrate over a control element $\Tm_i$, and apply the divergence theorem to yield
\begin{equation}
\label{eqn:EulerStepExact}
 \int_{\Tm_i} q^{n+1} \varphi^{(\ell)} d{\bf x} = 
 \int_{\Tm_i} q^n\varphi^{(\ell)} d{\bf x} 
    - \dt \int_{\Tm_i} \nabla \varphi^{(\ell)} \cdot {{\bf F}}_T(q^n) \, d{\bf x}
    + \dt \oint_{\partial \Tm_i} \varphi^{(\ell)} \, {{\bf F}}_T( q^n ) \cdot {\bf \hat{n}} \, ds,
\end{equation}
where ${\bf \hat{n}}$ is the outward pointing unit normal to element $\Tm_i$,
which reduces to
\begin{equation}
\label{eqn:EulerStepExactGalerkin}
  Q^{(\ell) \, n+1}_i = Q^{(\ell) \, n}_i
    - \frac{\dt}{|\Tm_i|} \underbrace{ \int_{\Tm_i}  \nabla \varphi^{(\ell)} \cdot
    {{\bf F}^h_T} \, d{\bf x}}_{ 
        \text{Interior} }
    \, + \, \frac{\dt}{|\Tm_i|} \underbrace{ \oint_{\partial \Tm_i} \varphi^{(\ell)} \, {\bf F}^{h \, *}_T  \cdot {\bf \hat{n}} \, ds }_{
        \text{ Edges } } 
\end{equation}
by orthogonality of the basis $\varphi$.  
In practice, both the interior and edge integrals are approximated
by appropriate numerical quadrature rules.
The flux values, ${\bf F}^{h \, *}_T$, in the edge integrals
still need to be defined.

\item[{\bf Step 4.}] Along each edge solve Riemann problems at each quadrature point by
using the left and right interface values.  In this work we use the 
well-known Lax-Friedrichs flux:
\begin{equation}
\label{eqn:LLF}
{\bf F}^{h \, *}_T\left(q^h_-, \, q^h_+ \right) 
\cdot {\bf \hat{n}} = \frac{1}{2} \left[ {\bf \hat{n}} \cdot \left( {\bf F}_T \left(q^h_+\right) + {\bf F}_T\left(q^h_-\right) \right)
- s \left( q^h_+ - q^h_- \right) \right],
\end{equation}
where $s$ is an estimate of the maximum global wave speed,
$q^h_-$ is the approximate solution evaluated on the element boundary
on the interior side of $\Tm_i$, and $q^h_+$ 
is the approximate solution evaluated on the element boundary
on the exterior side of $\Tm_i$.

\end{description}

\subsection{\it Boundary conditions}
\label{subsec:boundary}

In order to achieve high-order accuracy at the boundaries of the computational domain, a careful treatment of the solution in each boundary element is required. In particular, all simulations
in this work require either reflective (hard surface) or transparent (outflow) boundary conditions. 

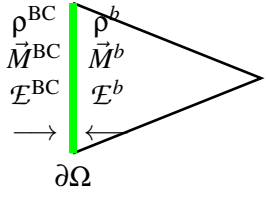
\begin{figure}
\begin{center}
\begin{tikzpicture}[scale=1, line width=1]
\coordinate (O) at (0,0);
\coordinate (A) at (2.5,1);
\coordinate (B) at (0,2);
\coordinate[label=left:$\begin{matrix}
\rho^{\text{BC}} \\
\vec{M}^{\text{BC}} \\
\E^{\text{BC}} \\
\longrightarrow
\end{matrix}$](c) at ($ (B)!.5!(O) $);
\coordinate[label=right:$\begin{matrix}
\rho^{b} \\
\vec{M}^{b} \\
\E^{b} \\
\longleftarrow
\end{matrix}$](c) at ($ (B)!.5!(O) $);
\coordinate[label=below:$\partial \Omega$](c) at ($ (O) $);
\draw (O)--(A)--(B);

\draw[line width=3,green] (B)--(O);
\end{tikzpicture}
\end{center}
\caption{Interior, $q^b$, and exterior, $q^{\text{BC}}$,
solution values on either side of the boundary
$\partial \Omega$. \label{fig:bcs}}
\end{figure}

Suppose the solution takes on the value 
\begin{equation}
q^b= \bigl(\rho^b, \, \vec{M}^b, \, \E^b \bigr)
\end{equation}
at a quadrature point ${\bf x}^b$ on the boundary $\partial \Omega$, and we wish
to define a boundary value, $q^{\text{BC}}$, on the exterior side of the boundary
$\partial \Omega$ that yields a flux with one of two desired boundary conditions.
This is depicted in Figure \ref{fig:bcs}. 
Let ${\bf \hat{t}}$ and ${\bf \hat{n}}$ be the unit tangent and unit normal
vectors to the boundary at the boundary point ${\bf x}^b$, respectively.
In both the reflective and transparent boundary conditions, we enforce continuity of the tangential components:
\begin{equation}
\label{eqn:bndy-tangent}
	\vec{M}^{\text{BC}} \cdot \hat{ \bf t } =  \vec{M}^b \cdot \hat{ \bf t }.
\end{equation}
The only difference between the two types of boundary conditions we consider lie in the normal direction. We set
\begin{equation}
\label{eqn:bndy-normal}
\vec{M}^{\text{BC}} \cdot \hat{ \bf n } =  \mp \vec{M}^b \cdot \hat{ \bf n },
\end{equation}
where the minus sign corresponds to the reflective boundary condition
and the plus sign corresponds to the transparent boundary condition.

From this we can easily write out the full boundary conditions 
at the point ${\bf x}^{\text{BC}}$:
\begin{equation}
	\begin{pmatrix}
	\rho^{\text{BC}} \\ \vec{M}^{\text{BC}} \\ \En^{\text{BC}}
	\end{pmatrix} 
	= 
	\begin{pmatrix}
	\rho^b \\ \left( \vec{M}^b \cdot \hat{ \bf t } \right) \hat{ \bf t } 
	\mp \left( \vec{M}^b \cdot \hat{ \bf n } \right) \hat{ \bf n } \\ \En^b
	\end{pmatrix},
\end{equation}
where again the minus sign corresponds to the reflective boundary condition
and the plus sign corresponds to the transparent boundary condition.

In order to achieve high-order time accuracy we need apply the above
boundary conditions to the time derivatives on the boundary:
\begin{equation}
	\begin{pmatrix}
	\rho^{\text{BC}}_{,t} \\ \vec{M}^{\text{BC}}_{,t} \\ \En^{\text{BC}}_{,t}
	\end{pmatrix} 
	= 
	\begin{pmatrix}
	\rho^b_{,t} \\ \left( \vec{M}^b_{,t} \cdot \hat{ \bf t } \right) \hat{ \bf t } \mp \left( \vec{M}^b_{,t} \cdot \hat{ \bf n } \right) \hat{ \bf n }
	   \\ \En^b_{,t}
	\end{pmatrix}, \quad
		\begin{pmatrix}
	\rho^{\text{BC}}_{,t,t} \\ \vec{M}^{\text{BC}}_{,t,t} \\ \En^{\text{BC}}_{,t,t}
	\end{pmatrix} 
	= 
	\begin{pmatrix}
	\rho^b_{,t,t} \\ 
	  \left( \vec{M}^b_{,t,t} \cdot \hat{ \bf t } \right) \hat{ \bf t } 
	  \mp \left( \vec{M}^b_{,t,t} \cdot \hat{ \bf n } \right) \hat{ \bf n } \\ \En^b_{,t,t}
	\end{pmatrix},
\end{equation}
where all time derivatives must be replaced by spatial derivatives using the PDE.

\section{Positivity preservation}
\label{sec:positivity}

The temporal evolution described in the previous section fails to retain
positivity of the solution, even in the simple case of linear advection with
smooth solutions that are near zero.  In fact, in extreme cases the projection of the
initial conditions can fail to retain positivity of the solution due to 
the Gibbs phenomena.  The positivity-preserving limiter we present follows a two
step procedure:
\begin{description}
\item[{\bf Step 1.}] Limit the moments in the expansion so that the solution is positive at
each quadrature point.
\item[{\bf Step 2.}] Limit the fluxes so that the cell averages retain positivity after a
single time step.
\end{description}
We now describe the first of these two steps.

\subsection{\it Positivity at interior quadrature points via moment limiters}
\label{sec:zh-limiter}

This section describes a procedure that implements the following: if the cell
averages are positive, then the solution is forced
to be positive at a preselected and finite collection of
quadrature points.  We select only the quadrature points
that are actually used in the numerical update; this includes
internal Gauss quadrature points as well as face/edge Gauss
quadrature points. Unlike other positivity limiting schemes \cite{ZhangShuTrimesh,article:ZhangShu10rectmesh} in the 
SSP Runge-Kutta framework, this step is not strictly necessary 
 to guarantee positivity of the cell average at the
next time-step; however, the main reason for applying the limiter
at quadrature points is to guarantee that each term in the update
is physical, which will reduce the total amount of additional 
limiting of the cell average updated needed in Section \ref{subsec:mpp-limiter}.
The process to maintain positivity at quadrature points is carried out in
a series of three simple steps.  Because this part of the limiter is entirely local, we focus on a single element $\Tm_k$, and therefore drop the subscript for ease of notation.

\subsubsection{\bf Step 0: Assume positivity of the cell averages.}
We assume that the cell averages for the density satisfies
$\avg{\rho}^n \geq \eps_0$, where $\eps_0 > 0$ is a cutoff parameter that
guarantees hyperbolicity of the system.  In this work, we set $\eps_0 =
10^{-12}$ in all simulations.  Furthermore, we assume that the cell average for the pressure
satisfies $\avg{p}^{n} > 0$, where the (average) pressure
$\avg{p}^n$ is defined through the averages of the other conserved quantities in \eqref{eqn:ave_press}.
Note that these two conditions are sufficient to imply that
the average energy $\avg{\E}^n$ is positive.


\subsubsection{\bf Step 1: Enforce positivity of the density.}

In this step, we enforce positivity of the density at each quadrature point
${\bf x}_m \in \Tm_k$.  Because $\avg{\rho}^n \geq \eps_0$, there exists a
(maximal) value $\theta^\rho \in [0,1]$ such that
\begin{equation}
\label{eqn:rho-theta}
    \rho^\theta_m := \avg{ \rho }^n + \theta \sum_{\ell=2}^{\mleg(\deg)} \rho^{(\ell)n} \, \varphi^{(\ell)}( {\bf x}_m )  \geq \epsilon_0 
\end{equation}
for all quadrature points ${\bf x}_m \in \Tm_k$ and all $\theta \in [0,
\theta^\rho]$.  Note that we drop the subscript that indicates the element number to
ease the complexity of notation for the the ensuing discussion.




\subsubsection{\bf Step 2: Enforce positivity of the pressure.}

Recall that the pressure is defined through the relation \eqref{eqn:eos}.
We seek to guarantee that $p({\bf x}_m) > 0$ for each quadrature point ${\bf x}_m \in
\Tm_i$.  In place of working directly with the pressure, we observe that it
suffices to guarantee that the product of the density and pressure is
positive.  To this end, we expand the momentum and energy in the same free
parameter $\theta$.  Similar to the density in \eqref{eqn:rho-theta},
we write the momentum and energy as
\begin{equation}
\label{eqn:momentum-rho}
    \vec{ M }^\theta_m := \avg{ \vec{ M } }^{ n } + 
    \theta
    \sum_{\ell=2}^{\mleg(\deg)} \vec{ M }^{(\ell)n} \, \varphi^{(\ell)}( {\bf x}_m ) 
    \quad \text{and} \quad
    \E^n_{m}   := \avg{ \E   }^{n} + \theta \sum_{\ell=2}^{\mleg(\deg)} 
    \E^{(\ell)n} \, \varphi^{(\ell)}( {\bf x}_m ) .
\end{equation}
We define deviations from cell averages as
\begin{equation}
   \left( \widetilde{\rho}^n_{m}, \, 
   \widetilde{ \vec{M} }_{m}, \, 
    \widetilde{\E}^n_{m} \right) :=  
    \sum_{\ell=2}^{\mleg(\deg)} 
    \left( \rho^{(\ell)n}, \, \vec{M}^{(\ell)n}, \, \E^{(\ell)n} \right)
     \varphi^{(\ell)}( {\bf x}_m ),
\end{equation}
and compactly write the expression for the limited variables as 
the cell average plus deviations:
\begin{equation}
\label{eqn:q-theta}
    q^\theta_m := \avg{q} + \theta \widetilde{ q }_m, \quad \text{where} \quad
    \avg{q}:=\left( \avg{\rho}^{ n }, \,
    \avg{ \vec{ M } }^{ n }, \, \avg{\E}^{ n } \right)
    \quad \text{and} \quad
    \widetilde{ q }_m = \left( \widetilde{\rho}^n_m, \, \widetilde{ \vec{M} }^n_{m},
    \, \widetilde{\E}^n_{m} \right).
\end{equation}

The product of the density and pressure at each quadrature point is
a quadratic function of $\theta$:
\begin{equation}\label{eq:rhopress}
     (\rho p)^\theta_m := \rho^\theta_m \, p^\theta_m = (\gamma-1)\left( 
        \E^\theta_m \rho^\theta_m - \frac{1}{2} \norm{ \vec{M}^\theta_m }^2
    \right),
\end{equation}
where $\rho^\theta_m$ is defined in  \eqref{eqn:rho-theta}.
%
%
After expanding each of these conserved variables in the
\emph{same scaling parameter} $\theta$,
we observe that
\begin{equation}
\label{eqn:product-rho-p}
\begin{aligned}
    \left(\rho p \right)^\theta_m
        &= (\gamma - 1)\left( \E^\theta_{m} \rho^\theta_{m} - \frac{1}{2} \norm{ \vec{M}^\theta_{m} }^2 \right)\\
        &= (\gamma-1)\left[
        a_m \theta^2 + b_m \theta +  
        \left( \underbrace{  
            \avg{ \E }^{n} \avg{ \rho }^{n} - 
            \frac{1}{2} \norm{ \avg{ \vec{M} }^{n} }^2
        }_{> 0} \right)
        \right],
\end{aligned}
\end{equation}
where $a_m$ and $b_m$ depend only on the quadrature point and higher-order terms of the expansions of 
density, energy, and momentum:
 \begin{equation}
    a_m = \widetilde{\E}^n_{m}\, \widetilde{\rho}^n_{m} -
        \frac{1}{2}  \norm{ \widetilde{ \vec{M} }^n_m }^2  \quad \text{and}
        \quad
    b_m =\widetilde{\E}^n_{m} \, \avg{ \rho }^{n}_m + \avg{ \E }^{n}_m \, \widetilde{\rho}^n_{m}- \widetilde{ \vec{M} }^n_{m} \cdot \avg{ \vec{M} }^n.
\end{equation}

The quadratic function defined by \eqref{eqn:product-rho-p} is non-negative
for at least one value of $\theta$, namely $\theta = 0$. However, 
if \eqref{eqn:product-rho-p} is positive at $\theta = 0$, then
we are guaranteed that there exists a $\theta_m \in
(0,1]$ that guarantees
$\left(\rho p \right)^\theta_m \geq 0 $ for all $\theta \in [0, \theta_m]$.
In particular, we are interested in finding the largest such $\theta$ (i.e.,
the least amount of damping).
Instead of exactly computing the optimal $\theta$, which could readily be done, but would require additional floating point operations, we make use of the following lemma \cite{seal2014explicit} to find an approximately optimal $\theta$.
\begin{lem}
The pressure function is a convex function of $\theta$ on $[0, \theta^\rho]$.
That is, 
\begin{equation}
\label{eqn:pressure-concavity}
p^{ \alpha \theta_1 + (1-\alpha) \theta_2 }_m
\geq
\alpha p^{\theta_1}_m + (1-\alpha) p^{ \theta_2}_m
\end{equation}
for all \mbox{$\theta_1, \theta_2 \in [0, \theta^\rho]$}, and $\alpha \in [0,1]$.
\end{lem}
\begin{proof}
We observe that directly from the definition of the limiter for the conserved
variables in  \eqref{eqn:q-theta} that
\begin{equation}
    q^{\alpha \theta_1 + (1-\alpha)\theta_2}_m 
    = \alpha q^{\theta_1}_m + (1-\alpha) q^{\theta_2}_m,
\end{equation}
and therefore
\begin{equation}
    p^{ \alpha \theta_1 + (1-\alpha) \theta_2 }_m
    = p\left( q^{\alpha \theta_1 + (1-\alpha) \theta_2}_m \right) 
        = p\left( \alpha q^{\theta_1}_m + (1-\alpha) q^{ \theta_2 }_m \right)
        \geq
        \alpha p^{\theta_1}_m + (1-\alpha) p^{ \theta_2}_m.
\end{equation}
The final inequality follows because $\rho^\theta_m > 0$ for all $0 \leq \theta
\leq \theta^\rho$, and the pressure is a convex function (of the conserved variables) whenever the density is positive.
\end{proof}

As a consequence of  \eqref{eqn:pressure-concavity}  we can define
\begin{equation}
\theta_m := \min \left( \frac{ p^{0}_m }{ p^{0}_m - p^{\theta^\rho}_m },
\, \theta^\rho_m \right),
\end{equation}
which will guarantee that $p^\theta_m > 0$ for all $\theta \in [0, \theta_m]$.
Finally, we define the scaling parameter for the entire cell as
\begin{equation}
    \theta := \min_m \left\{ \theta_{m} \right\}
\end{equation}
and use this value to limit the higher order coefficients in the Galerkin
expansions of the density, momentum, and energy 
displayed in Eqns.
\eqref{eqn:rho-theta} and \eqref{eqn:momentum-rho}.
This definition gives us the property that $\rho^n_m \geq \eps_0$ and $p^n_{m} >0$ at each quadrature
point ${\bf x}_m \in \Tm_i$.  This process is repeated (locally) in each element $\Tm_i$ in
the mesh.  As a side benefit to guaranteeing that the density and pressure are
positive, we have the following remark.
\begin{rmk}
If $\rho^\theta_m$ and $p^\theta_m$ are positive at each quadrature point,
then $\E^\theta_m$ is also positive at each quadrature point.
\end{rmk}
\begin{proof}
Divide \ \eqref{eqn:product-rho-p} by $(\gamma-1)\rho^\theta_m$ and add
$\frac{1}{2} \norm{ \vec{M}_m^\theta }^2$ to both sides.
\end{proof}

This concludes the first of two steps for retaining positivity of the
solution.  We now move on to the second and final step, which takes into
account the temporal evolution of the solver.

\subsection{\it Positivity of cell averages via parameterized flux limiters}
\label{subsec:mpp-limiter}
The procedure carried out for the flux limiter presented in this section is very similar to recent work
for finite volume \cite{christlieb2015high} as well as finite difference
\cite{seal2014explicit,ChFeSeTa2015} methods.  When compared to the finite
difference methods, the main difference in this discussion is that the
expressions do not simplify as much because quantities such as the edge lengths
must remain in the expressions.  This makes them more similar to 
work on finite
volume schemes \cite{christlieb2015high}. Overall however, there is little difference between flux limiters on Cartesian and
unstructured meshes, and between flux limiters for finite difference, finite volume (FV), and discontinuous Galerkin (DG) schemes.
This is because the updates for the cell average in a DG solver can be
made to look identical to the update for a FV solver, and once flux interface values are identified, a conservative FD method can be made to look like a FV
solver, albeit with a different stencil for the discretization.

All of the aforementioned papers rely on the result of Perthame and Shu
\cite{PerthameShu}, which states that a first-order finite
volume scheme (i.e., one that is based on a piecewise constant 
representation with forward Euler time-stepping) that uses
the Lax-Friedrichs (LxF) numerical flux is positivity-preserving under
the usual CFL condition. 
Similar to previous work, we leverage this idea and incorporate it
into a flux limiting procedure.  Here, the focus is on Lax-Wendroff discontinuous Galerkin schemes.  

In this work we write out the details of the limiting procedure
only for the case of 2D triangular elements. However, 
all of the formulas generalize
to higher dimensions and Cartesian meshes.

To begin, we consider the Euler equations
\eqref{eqn:euler-eqns}
and a mesh that fits the description given in \S\ref{sec:dg_spatial}. 
After integration over a single cell, $\Tm_i$, and an application of the
divergence theorem, we see that the exact evolution equation for the cell
average of the density is given by  
\begin{equation}
    \label{eqn:semi-discrete}
    \frac{d}{dt} \int_{\Tm_i} q \, d{\bf x}
        = - \oint_{\p \Tm_i } \vec{F} \cdot {\bf \hat{n}} \, ds,
\end{equation}
where ${\bf \hat{n}}$ is the outward pointing (relative to $\Tm_i$) unit normal 
to the boundary of $\Tm_i$.
Applying to this equation a first-order finite volume
discretization using the Lax-Friedrichs flux yields
\begin{equation}
	\avg{q}^{n+1}_i = \avg{q}^{n}_i - \frac{\dt}{|\Tm_i|}
	\sum_{e \in {\Tm_i}}
	f^{ \, \text{LxF}}_{e},
\end{equation}
and the Lax-Friedrichs flux is
\begin{equation}
f^{ \, \text{LxF}}_{e} :=
 \frac{1}{2} \vec{n}_e \cdot \left( \vec{F}\left(\avg{q}^n_{e+}\right) +
\vec{F}\left(\avg{q}^n_{i}\right) \right)
- \frac{1}{2} \| \vec{n}_e \| s \left( \avg{q}^n_{e+} - \avg{q}^n_{i} \right),
\end{equation}
where $\vec{n}_e$ is an outward-pointing (relative to $\Tm_i$) 
normal vector to edge $e$
with the property that $\| \vec{n}_e \|$ is the length of edge $e \in {\Tm_i}$,
and the $e+$ index refers to the solution on edge $e$ on the exterior side of
${\Tm_i}$.
We use a \emph{global} wave speed $s$, because
this flux defines a provably positivity-preserving
scheme \cite{PerthameShu}.  Other fluxes can be used, provided they are
positivity-preserving (in the mean).

Recall that the update for the LxW-DG method is given in
\eqref{eqn:EulerStepExactGalerkin}.  The numerical edge flux,
${\bf F}^{h*}_T$, is based on 
the temporal Taylor series expansion of the fluxes.  The update
for the cell averages takes the form:
\begin{equation}
  \avg{q}^{n+1}_i = \avg{q}^n_i
    - \frac{\dt}{|\Tm_i|} \oint_{\partial \Tm_i} {\bf F}^{h*}_T \cdot 
    \vec{\hat{n}} \, ds,
\end{equation}
which in practice needs to be replaced by a numerical quadrature
along the edges. Applying Guassian quadrature along each edge
produces the following edge value (or face value in the case of 3D):
\begin{equation}
f_e^{\text{LxW}} :=
    \frac{1}{2} \sum_{k=1}^{{M}_{\text Q}} \omega_k \, {\bf F}_T^{h*}( {\bf x}_k ) \cdot {\bf n}_e,
\end{equation}
where ${\bf x}_k$ and ${\omega}_k$ are Gaussian quadrature points and weights,
respectively, for integration along edge $e$.
%
%
This allows us to write the update for the cell average in the Lax-Wendroff DG
method in a similar fashion to that of the Lax-Friedrichs solver, but this time we have higher-order fluxes:
\begin{equation}
    \avg{q}^{n+1}_i = \avg{q}^n_i - \frac{\dt}{|\Tm_i|} \sum_{e \in {\Tm_i} } f^{\, \text{LxW}}_e.
\end{equation}


Next we define a parameterized flux on edge $e$ by
\begin{equation}
    \widetilde{f}_e:=\theta_e \left( f^{\, \text{LxW}}_e - 
    f^{\, \text{LxF}}_e \right)+{f}^{\text{LxF}}_e, 
\end{equation}
where $\theta_e \in [0,1]$ is as free parameter that is yet to be
determined.
We also define the quantity, $\Gamma_i$, which by virtue of the positivity
of the Lax-Friedrichs method is positive in both density and pressure:
\begin{equation}
\avg{q}^{n+1}_i = \frac{\dt}{|\Tm_i|} \Gamma_i, \quad \text{where} \quad
	 \Gamma_i := \frac{|\Tm_i|}{\dt} \avg{q}^n_i -  \sum_{e \in \Tm_i} 
	 {f}^{\text{LxF}}_e.
\end{equation}
In order to retain a positive density in the high-order
update formula, the following condition must be satisfied:
\begin{equation}
	\avg{\rho}^{n+1}_i = \avg{\rho}^n_i - \frac{\dt}{|\Tm_i|} \sum_{e \in 
	\Tm_i} \tilde{f}^{\rho}_e \ge 0
	\quad \Longrightarrow \quad
	\avg{\rho}_i^{n+1} = \frac{\dt}{|\Tm_i|} \left( \Gamma_{i}^{\rho} + 
     \sum_{e \in \Tm_i}
     \theta_e  \Delta f^\rho_e \right) \geq 0,
\end{equation}
where 
\begin{equation}
\Delta f_e := f^{\text{LxF}}_e -f^{\text{LxW}}_e.
\end{equation}
Note that a $\rho$ superscript is introduced to the flux function
in order to denote the first component of the flux, namely the 
mass flux.
Positivity of the density is achieved if 
\begin{equation}
\label{eqn:low-order-update-eqn}
	\sum_{e \in \Tm_i} \theta_e \Delta f^{\rho}_e \ge -\Gamma^\rho_i.
\end{equation}

The basic procedure for positivity limiting is to reduce the values of $\theta_e$ until  inequality condition  \eqref{eqn:low-order-update-eqn} 
is satisfied. For a 
triangular mesh, there are three values of $\theta_e$ that contribute to
each cell. In this case there exists a three-dimensional feasible region 
that contains all admissible $\theta_e$ values, one for each $e \in \Tm_i$,
where the new average density, $\avg{\rho}_i^{ \, n+1}$, is positive.
Finding the exact boundary of this set is computationally impractical; and therefore, we make an approximation so that the problem becomes 
much simpler. In particular, we approximate the feasible region by a rectangular
cuboid:
\begin{equation}
 	S_i^\rho := \left[0,\varLambda_{e_{i1}} \right] \times \left[0,\varLambda_{e_{i2}}]\times[0,\varLambda_{e_{i3}}\right] 
	\subseteq [0,1]^3,
\end{equation}
where $e_{i1}$, $e_{i2}$, and $e_{i3}$ are the three edges that make up element $\Tm_i$,
over which
\begin{equation}
	\avg{ \rho }^{n+1}_i = \avg{\rho}^n_i-\frac{\dt}{|\Tm_i|} 
	\left(  \widetilde{f}^{ \, \rho}_{e_{i1}} + \widetilde{f}^{ \, \rho}_{e_{i2}} 
	+ \widetilde{f}^{ \, \rho}_{e_{i3}} \right)
	 \ge 0,
	\quad \forall
	 \left(\theta_{e_{i1}}, \, \theta_{e_{i2}}, \,  \theta_{e_{i3}}\right) \in S_i^\rho.
\end{equation}

Once we have determined the feasible region for each element over which the average density,
 $\avg{ \rho }^{n+1}_i$, remains positive, the next step is to rescale each
 $\varLambda_{e_{ik}}$ for $k=1,2,3$ to also guarantee a positive average pressure,
 $\avg{ p }^{n+1}_i$, on the same element:
 \begin{equation}
 	S_i^{\rho p} := \left[0,\mu_{e_{i1}} \, \varLambda_{e_{i1}} \right] 
	\times \left[0,\mu_{e_{i2}} \, \varLambda_{e_{i2}}]
	\times[0,\mu_{e_{i3}} \, \varLambda_{e_{i3}}\right] 
	\subseteq [0,1]^3,
\end{equation}
where $0 \le \mu_{e_{ik}} \le 1$ for each $i$ and for each $k=1,2,3$.

We leave the details of the procedure to determine the feasibility region
to the next subsection, in which
we also summarize the full algorithm.

\subsection{\it Putting it all together: An efficient implementation of the positivity-preserving limiter}
An efficient implementation of the positivity preserving limiter should avoid
communication with neighboring cells as much as possible.  Indeed, this is one
advantage of single-stage, single-step methods such as the Lax-Wendroff
discontinuous Galerkin method.  In order to avoid additional communication
overhead, we suggest the implementation described below.

In the formulas below we make use of the following quantities:
\begin{align}
{\bf n}_{e} &:= \text{normal vector to edge $e$ such that } \| {\bf n}_e \| 
\text{ is equal to the length of edge $e$}, \\
{e}_{ik} &:= \text{label of $k^{\text{th}}$ edge ($k=1,2,3$) of element $\Tm_i$}, \\
\label{eqn:epsik}
{\epsilon}_{ik} &:= \begin{cases}
+1 & \text{if } \, {\bf n}_{e_{ik}} \, \text{ is outward pointing relative to $\Tm_i$}, \\
-1 & \text{if } \, {\bf n}_{e_{ik}} \, \text{ is inward pointing relative to $\Tm_i$},
\end{cases} \\
i_{e} &:= \text{the element that has $e$ as an edge and for which ${\bf n}_{e}$
is outward pointing}, \\
j_{e} &:= \text{the element that has $e$ as an edge and for which ${\bf n}_{e}$
is inward pointing}.
\end{align}
These quantities are illustrated in Figure \ref{fig:edge}.

\begin{figure}
\begin{center}
(a) \begin{tikzpicture}[scale=1, line width=1]
\coordinate (O) at (0,0);
\coordinate (A) at (2.5,1);
\coordinate (B) at (0,2);
\coordinate (C) at (-2.5,1);
\coordinate (D1) at (0,1);
\coordinate (D2) at (0.75,1);
\draw (O)--(A)node[rotate=0,xshift=-1.2cm, yshift=0.0cm] {$\Tm_{j_e}$}--(B)--(O);
\draw (C)--(O)node[rotate=90,xshift=1cm, yshift=0.225cm] {edge:  $e$};
\draw (C)node[rotate=0,xshift=1.5cm, yshift=0.0cm] {$\Tm_{i_e}$}--(B);
\draw[->] (D1)node[xshift=0.5cm,yshift=-0.25cm]{$\vec{n}_e$}--(D2);
\end{tikzpicture}
\hspace{5mm}
(b) \begin{tikzpicture}[scale=1, line width=1]
\coordinate (CC) at (1.5,1);
\coordinate (O) at (0,0);
\coordinate (A) at (3,0);
\coordinate (B) at (1.5,2);
\coordinate (D1) at (1.5,-0.4);
\coordinate (D2) at (1.5,0.4);
\coordinate (E1) at (1.9,0.7375);
\coordinate (E2) at (2.6,1.2625);
\coordinate (F2) at (1.1,0.7375);
\coordinate (F1) at (0.4,1.2625);
\draw (O)--(A)--(B)--(O);
\draw[->] (D1)node[xshift=0.4cm,yshift=0.65cm]{$\vec{n}_{e_{i1}}$}--(D2);
\draw (CC)node{$\Tm_i$};
\draw[->] (E1)node[xshift=0.9cm,yshift=0.2cm]{$\vec{n}_{e_{i2}}$}--(E2);
\draw (CC)node{$\Tm_i$};
\draw[->] (F1)node[xshift=-0.1cm,yshift=-0.3cm]{$\vec{n}_{e_{i3}}$}--(F2);
\draw (CC)node{$\Tm_i$};
\end{tikzpicture}
\end{center}
\caption{Illustrations of various quantities needed in the
discontinuous Galerkin update on unstructured grids.
Panel (a) illustrates the normal vector $\vec{n}_e$ to the edge
$e$ and the two elements $i_e$ (on the outward side of $\vec{n}_e$)
and $j_e$ (on the inward side of $\vec{n}_e$). 
Panel (b) illustrates the three normal vectors $\vec{n}_{e_{i1}}$,
$\vec{n}_{e_{i2}}$, and $\vec{n}_{e_{i3}}$ to the three edges of 
element $\Tm_i$. In the case shown in Panel (b), the 
values of ${\epsilon}_{ik}$ as defined in Equation 
\eqref{eqn:epsik} are ${\epsilon}_{i1} = -1$,
${\epsilon}_{i2} = 1$, and ${\epsilon}_{i3} = -1$.
 \label{fig:edge}}
\end{figure}
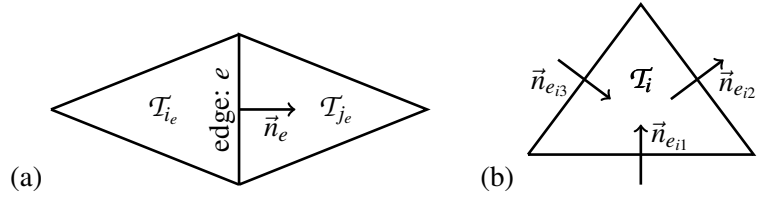

\medskip

\noindent
Loop over all elements, $i=1,2,\ldots,M_{\text{elems}}$:
\begin{description}
\item[\quad {\bf Step 1.}] Enforce positivity of the density and pressure at internal
and boundary quadrature points using the limiter described in detail in Section \ref{sec:zh-limiter}.
\item[\quad {\bf Step 2.}] Compute the time-averaged fluxes, ${\bf F}^n_{T}$, defined
in Equation \eqref{eqn:taylor3-flux}.
\end{description}

\medskip

\noindent
Loop over all edges, $e=1,2,\ldots,M_{\text{edges}}$:
\begin{description}
\item[\quad {\bf Step 3.}] For each quadrature point, ${\bf x}_\ell$,
on the current edge compute the high-order numerical flux:
\begin{align}
{\bf F}_T^{h*}( {\bf x}_{\ell} ) \cdot {\bf n}_e &:=  \frac{{\bf n}_e}{2} 
\cdot \left[ {\bf F}_T \hspace{-0.8mm} \left(q^h \hspace{-0.8mm} \left({\bf x}_{k}^{+}\right)\right)
+ {\bf F}_T \hspace{-0.8mm} \left(q^h \hspace{-0.8mm} \left({\bf x}_{k}^{-}\right)\right) \right]
- \frac{s \| {\bf n}_e  \|}{2}  \hspace{-0.6mm} \left( q^h \hspace{-0.8mm} \left({\bf x}_{k}^{+}\right) - 
q^h \hspace{-0.8mm} \left({\bf x}_{k}^{-}\right) \right),
\end{align}
where $s$ is an estimate of the maximum global wave speed.
From these numerical flux values, compute the edge-averaged high-order flux:
\begin{align}
f_e^{\text{LxW}} &:=
    \frac{1}{2} \sum_{\ell=1}^{{M}_{\text Q}} \omega_{\ell} \, {\bf F}_T^{h*}( {\bf x}_{\ell} ) \cdot {\bf n}_e,
\end{align}
where $\omega_k$ are the weights in the Gauss-Legendre numerical quadrature
with ${M}_{\text Q}$ points.
Still on the same edge, also compute the  edge-averaged low-order flux:
\begin{align}
	f^{ \, \text{LxF}}_{e} &:=
 \frac{\vec{n}_e}{2}  \cdot \Bigl[ \vec{F}\left(\avg{q}^{\, h}_{j_e}\right) +
\vec{F}\left(\avg{q}^{\, h}_{i_e}\right) \Bigr]
- \frac{s \| \vec{n}_e \|}{2} \Bigl( \avg{q}^{\, h}_{j_e} - \avg{q}^{\, h}_{i_e} \Bigr).
\end{align}
Finally, set
\begin{equation}
\Lambda_e = 1.
\end{equation}
\end{description}

\medskip

\noindent
Loop over all elements, $i=1,2,\ldots,M_{\text{elems}}$:
\begin{description}
\item[\quad {\bf Step 4.}] Let $\Delta f_{k} = 
\epsilon_{ik} \left( f^{\text{LxF}}_{e_{ik}} -f^{\text{LxW}}_{e_{ik}}\right)$ 
for $k=1,2,3$ represent the high-order flux contribution on the three edges of
element $\Tm_i$. We choose a reordering of the indices 1, 2, 3 
to the indices $a$, $b$, $c$ such that the three
flux differences are ordered as follows:
\begin{equation}
\Delta f^{\rho}_{{a}} \le \Delta f^{\rho}_{{b}} \le \Delta f^{\rho}_{{c}}.
\end{equation}
We now define the three values $\Lambda_{ia}$,
$\Lambda_{ib}$, and $\Lambda_{ic}$ and
consider four distinct cases.

\medskip

\begin{description}
\item[\quad {\bf Case 1.}] If
\, $0 \le \Delta f^{\rho}_{a} \le \Delta f^{\rho}_{b} \le \Delta f^{\rho}_{c}$, \,
then
\begin{equation}
 \Lambda_{ia} =\Lambda_{ib} =\Lambda_{ic} = 1.
\end{equation}

\item[\quad {\bf Case 2.}] If \, $\Delta f^{\rho}_{a} < 0 \le \Delta f^{\rho}_{b} \le \Delta f^{\rho}_{c}$, \, then set
\begin{equation}
\varLambda_{{ia}} = \min\left\{
1, \frac{\Gamma_i}{\left|\Delta f^{\rho}_{{a}}\right|}\right\}, \quad
\varLambda_{{ib}} = \varLambda_{{ic}} = 1.
\end{equation}

\item[{\quad \bf Case 3.}] If \, $\Delta f^{\rho}_{{a}}  \le \Delta f^{\rho}_{{b}} < 0 \le 
 \Delta f^{\rho}_{{c}}$, \, then set
\begin{equation}
\varLambda_{{ia}} = \varLambda_{{ib}} =  \min\left\{1, \frac{\Gamma_i}{\left|\Delta f^{\rho}_{{a}}+\Delta f^{\rho}_{{b}}\right|}\right\}, \quad
\varLambda_{{ic}} = 1.
\end{equation}

\item[\quad {\bf Case 4.}] If \, $\Delta f^{\rho}_{{a}}  \le \Delta f^{\rho}_{{b}}  \le 
 \Delta f^{\rho}_{{c}} < 0$, \, then set
\begin{equation}
\varLambda_{{ia}}
= \varLambda_{{ib}} = 
\varLambda_{{ic}}  =  \min\left\{ 1, \frac{\Gamma_i}{\left|\Delta f^{\rho}_{{a}}+\Delta f^{\rho}_{{b}}+\Delta f^{\rho}_{{c}}\right|}\right\}.
\end{equation}
\end{description}
Note that in each of these cases the ratios used in the relevant formulas
are found by setting the positive contributions on the left-hand side of inequality 
\eqref{eqn:low-order-update-eqn} equal to zero, and then solving for the
remaining elements.  This is equivalent to only looking at the worst case
scenario where mass is only allowed to flow out of element $\Tm_i$.

\medskip 

\item [\quad {\bf Step 5.}] Define the Lax-Friedrichs average solution:
\begin{equation}
\avg{q}^{ \, \text{LxF}}_i := \avg{q}^n_i-\frac{\dt}{|\Tm_i|} 
\sum_{k=1}^3 f^{ \, \text{LxF}}_{{k}},
\end{equation}
and do the following.

\begin{description}
\item[{\bf (a)}] Loop over all seven cases, $c=111, 110, 101, 100, 011, 010, 001$,
 shown in Figure \ref{fig:seven-cases} and for each case construct the average solution:
\begin{equation}
\avg{q}^{ \, c}_i := \avg{q}^n_i-\frac{\dt}{|\Tm_i|} 
\sum_{k=1}^3 \alpha_k  \, \varLambda_{{ik}} \, f^{ \, \text{LxW}}_{k}.
\end{equation}
The seven cases studied here enumerate all of the possible values of $\alpha_k
\in \{0,1\}$, with the exception of $c=000$, which reduces to updating the
element with purely a Lax-Friedrichs flux.

\item[{\bf (b)}] For each $c$, 
determine the largest value of $\mu_c \in [0,1]$ such that
the average pressure as defined by \eqref{eqn:ave_press}
and based on the average state:
\begin{equation}
\mu_c \, \avg{q}^{ \, c}_i + \left( 1 - \mu_c \right) \, \avg{q}^{ \, \text{LxF}}_i
\end{equation}
is positive. Note that the average pressure is always positive when
$\mu_c=0$, and if the average pressure is positive for $0 \leq r \leq 1$, then
the average pressure will be positive for any $0 \leq \mu_c \leq r$.  This is
because the pressure is a convex function of $\mu_c$.

\item[{\bf (c)}] Rescale the edge $\varLambda$ values (when compared to the
neighboring elements) based on $\mu_c$ as follows:
\begin{align}
	\varLambda_{e_{i1}} &= \min \left\{
	\varLambda_{e_{i1}}, \, \varLambda_{{i1}} \cdot \min\Bigl\{ \mu_{111}, \, \mu_{110},  \,\mu_{101}, \,
	\mu_{100} \Bigr\} \right\}, \\
	\varLambda_{e_{i2}} &= \min \left\{
	\varLambda_{e_{i2}}, \, \varLambda_{{i2}} \cdot \min\Bigl\{ \mu_{111}, \, \mu_{110},  \, \mu_{011}, \,
	\mu_{010} \Bigr\} \right\}, \\
	\varLambda_{e_{i3}} &= \min \left\{
	\varLambda_{e_{i3}}, \, \varLambda_{{i3}} \cdot \min\Bigl\{ \mu_{111}, \, \mu_{101},  \, \mu_{011},  \,
	\mu_{001} \Bigr\} \right\}.
\end{align}
\end{description}
\end{description}

\medskip

\noindent
Loop over all elements, $i=1,2,\ldots,M_{\text{elems}}$:
\begin{description}
\item [\quad {\bf Step 6.}] For each of the three edges that
make up element $\Tm_i$ determine the damping coefficients,
$\theta_k$ for $k=1,2,3$, as follows:
\begin{equation}
\theta_k = \Lambda_{e_{ik}} \quad \text{for} \quad k=1,2,3.
\end{equation}
Update the cell averages:
\begin{equation}
Q^{(1) \, n+1}_i = Q^{(1) \, n}_i - \frac{\dt}{|\Tm_i|} 
\sum_{k=1}^3 \Bigl[ 
\theta_{k} \, f^{ \, \text{LxW}}_{e_{ik}}
 + \left( 1- \theta_{{k}} \right) \, f^{ \, \text{LxF}}_{e_{ik}}
\Bigr],
\end{equation}
as well as the high-order moments
\begin{equation}
  Q^{(\ell) \, n+1}_i = Q^{(\ell) \, n}_i
    - \frac{\dt}{|\Tm_i|} { \int_{\Tm_i}  \nabla \varphi^{(\ell)} \cdot
    {{\bf F}^h_T} \, d{\bf x}}
    \, + \, \frac{\dt}{|\Tm_i|} { \oint_{\partial \Tm_i} \varphi^{(\ell)} \, {\bf F}^{h \, *}_T  \cdot {\bf \hat{n}} \, ds }
\end{equation}
for $2 \le \ell \le \mleg$, where exact integration is replaced by numerical quadrature.

\end{description}

\begin{figure}
\begin{center}
(a)\begin{tikzpicture}[scale=0.8, line width=1]
\coordinate (A) at (-1.5cm,-1.cm);
\coordinate (C) at (0.5cm,-1.0cm);
\coordinate (B) at (-1.5cm,1.0cm);
\draw (A) -- node[rotate=90,xshift=0.0cm, yshift=0.2cm] {$\alpha_1 = 1$} (B) -- 
node[right,rotate=-45,xshift=-0.5cm, yshift=0.2cm] { \, $\alpha_3 = 1$} (C) -- node[below] {$\alpha_2 = 1$} (A);
\end{tikzpicture}
\hspace{4mm}
(b)\begin{tikzpicture}[scale=0.8, line width=1]
\coordinate (A) at (-1.5cm,-1.cm);
\coordinate (C) at (0.5cm,-1.0cm);
\coordinate (B) at (-1.5cm,1.0cm);
\draw (A) -- node[rotate=90,xshift=0.0cm, yshift=0.2cm] {$\alpha_1 = 1$} (B) -- 
node[right,rotate=-45,xshift=-0.5cm, yshift=0.2cm] { \, $\alpha_3 = 0$} (C) -- node[below] {$\alpha_2 = 1$} (A);
\end{tikzpicture}
\hspace{4mm}
(c)\begin{tikzpicture}[scale=0.8, line width=1]
\coordinate (A) at (-1.5cm,-1.cm);
\coordinate (C) at (0.5cm,-1.0cm);
\coordinate (B) at (-1.5cm,1.0cm);
\draw (A) -- node[rotate=90,xshift=0.0cm, yshift=0.2cm] {$\alpha_1 = 1$} (B) -- 
node[right,rotate=-45,xshift=-0.5cm, yshift=0.2cm] { \, $\alpha_3 = 1$} (C) -- node[below] {$\alpha_2 = 0$} (A);
\end{tikzpicture}
\hspace{4mm}
(d)\begin{tikzpicture}[scale=0.8, line width=1]
\coordinate (A) at (-1.5cm,-1.cm);
\coordinate (C) at (0.5cm,-1.0cm);
\coordinate (B) at (-1.5cm,1.0cm);
\draw (A) -- node[rotate=90,xshift=0.0cm, yshift=0.2cm] {$\alpha_1 = 1$} (B) -- 
node[right,rotate=-45,xshift=-0.5cm, yshift=0.2cm] { \, $\alpha_3 = 0$} (C) -- node[below] {$\alpha_2 = 0$} (A);
\end{tikzpicture}

\vspace{3mm}

(e)\begin{tikzpicture}[scale=0.8, line width=1]
\coordinate (A) at (-1.5cm,-1.cm);
\coordinate (C) at (0.5cm,-1.0cm);
\coordinate (B) at (-1.5cm,1.0cm);
\draw (A) -- node[rotate=90,xshift=0.0cm, yshift=0.2cm] {$\alpha_1 = 0$} (B) -- 
node[right,rotate=-45,xshift=-0.5cm, yshift=0.2cm] { \, $\alpha_3 = 1$} (C) -- node[below] {$\alpha_2 = 1$} (A);
\end{tikzpicture}
\hspace{4mm}
(f)\begin{tikzpicture}[scale=0.8, line width=1]
\coordinate (A) at (-1.5cm,-1.cm);
\coordinate (C) at (0.5cm,-1.0cm);
\coordinate (B) at (-1.5cm,1.0cm);
\draw (A) -- node[rotate=90,xshift=0.0cm, yshift=0.2cm] {$\alpha_1 = 0$} (B) -- 
node[right,rotate=-45,xshift=-0.5cm, yshift=0.2cm] { \, $\alpha_3 = 0$} (C) -- node[below] {$\alpha_2 = 1$} (A);
\end{tikzpicture}
\hspace{4mm}
(g)\begin{tikzpicture}[scale=0.8, line width=1]
\coordinate (A) at (-1.5cm,-1.cm);
\coordinate (C) at (0.5cm,-1.0cm);
\coordinate (B) at (-1.5cm,1.0cm);
\draw (A) -- node[rotate=90,xshift=0.0cm, yshift=0.2cm] {$\alpha_1 = 0$} (B) -- 
node[right,rotate=-45,xshift=-0.5cm, yshift=0.2cm] { \, $\alpha_3 = 1$} (C) -- node[below] {$\alpha_2 = 0$} (A);
\end{tikzpicture}
\end{center}
\caption{Seven cases used to enforce the positivity of the average pressure
on each element: (a) $111$, (b) $110$, (c) $101$, (d) $100$, (e) $011$,
(f) $010$, and (g) $001$. \label{fig:seven-cases}}
\end{figure}
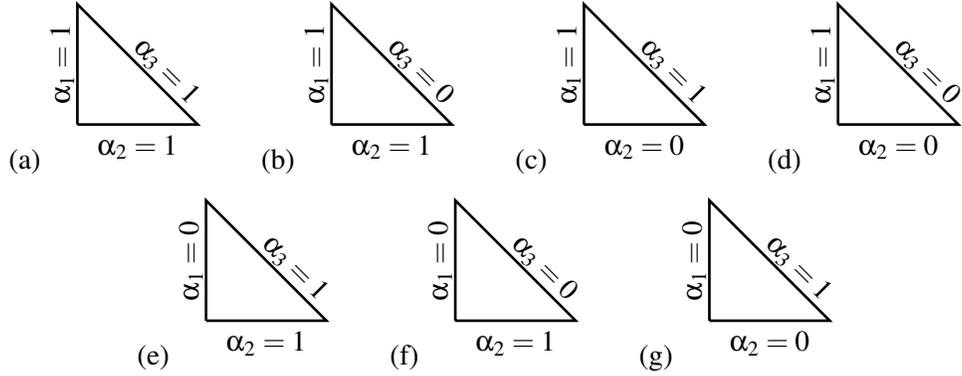

\begin{rmk}
Extensions to 2D Cartesian, 3D Cartesian, and 3D tetrahedral mesh elements
follow directly from what is presented here,
and require considering flux values along each of the edges/faces of a given element.
\end{rmk}


\section{Numerical results}
\label{sec:numerical-results}

\subsection{\it Implementation details}

All of the results presented in this section are implemented in the open-source
software package {\sc dogpack} \cite{dogpack}.  In addition, 
the positivity limiter described thus far is not designed to handle shocks, and
therefore an additional limiter needs to be applied in order to prevent spurious oscillations from developing (e.g., in problems that contain shocks but have large densities).  
There are many options available for this step, but 
in this work, we supplement the
positivity-preserving limiter presented here with the recent shock-capturing limiter developed in 
\cite{MoeRossSe15} in order to navigate shocks that develop in the solution. Specifically we use the version of this limiter
that works with the primitive variables, and we set the parameter $\alpha=500 \Delta x^{1.5}$. 
In our experience, this limiter with these parameters
offers a good balance between damping oscillations while maintaining sharply refined solutions.
Additionally, we point out that extra efficiency can be realized by locally storing quantities computed for the
aforementioned positivity-limiter as well as this shock-capturing limiter. 

Unless otherwise noted, these examples use a CFL number of $0.08$ with a $3^{\text{rd}}$ order
Lax-Wendroff time discretization.  
All of the examples have the positivity-preserving and shock-capturing
limiters turned on.

\subsection{\it One-dimensional examples}



In this section we present some standard one-dimensional problems that can be
found in \cite{seal2014explicit} and references therein. These problems are
designed to break codes that do not have a mechanism to retain positivity of
the density and pressure, but with this limiter, we are able to successfully
simulate these problems.

\subsubsection{\it Double rarefaction problem}
\label{ex:doublemach}

Our first example is the double rarefaction problem that can be found in
\cite{article:ZhangShu10rectmesh,ZhangShu11,seal2014explicit}. This is a Riemann problem with initial conditions
given by $(\rho_{L},u^1_{L},p_L)=(7,-1,0.2)$ and $(\rho_{R},u^1_{R},p_R)=(7,1,0.2)$. The solution involves two
rarefaction waves that move in opposite directions that leave near zero
density and pressure values in the post shock regime.
We present our solution on a mesh with a course resolution of $\Delta
x=\frac{1}{100}$, as well as a highly refined
solution with $\Delta x=\frac{1}{1000}$. Our results, shown in Figure \ref{fig:1ddoublemach}, are comparable to those obtained in other works, however the shock-capturing limiter
we use is not very diffusive and therefore there is a small amount of oscillation visible in the solution at the lower resolution. However
this oscillation vanishes for the more refined solution.
\begin{figure}
\begin{center}
 \subfigure[]{\includegraphics[width=0.45 \linewidth]{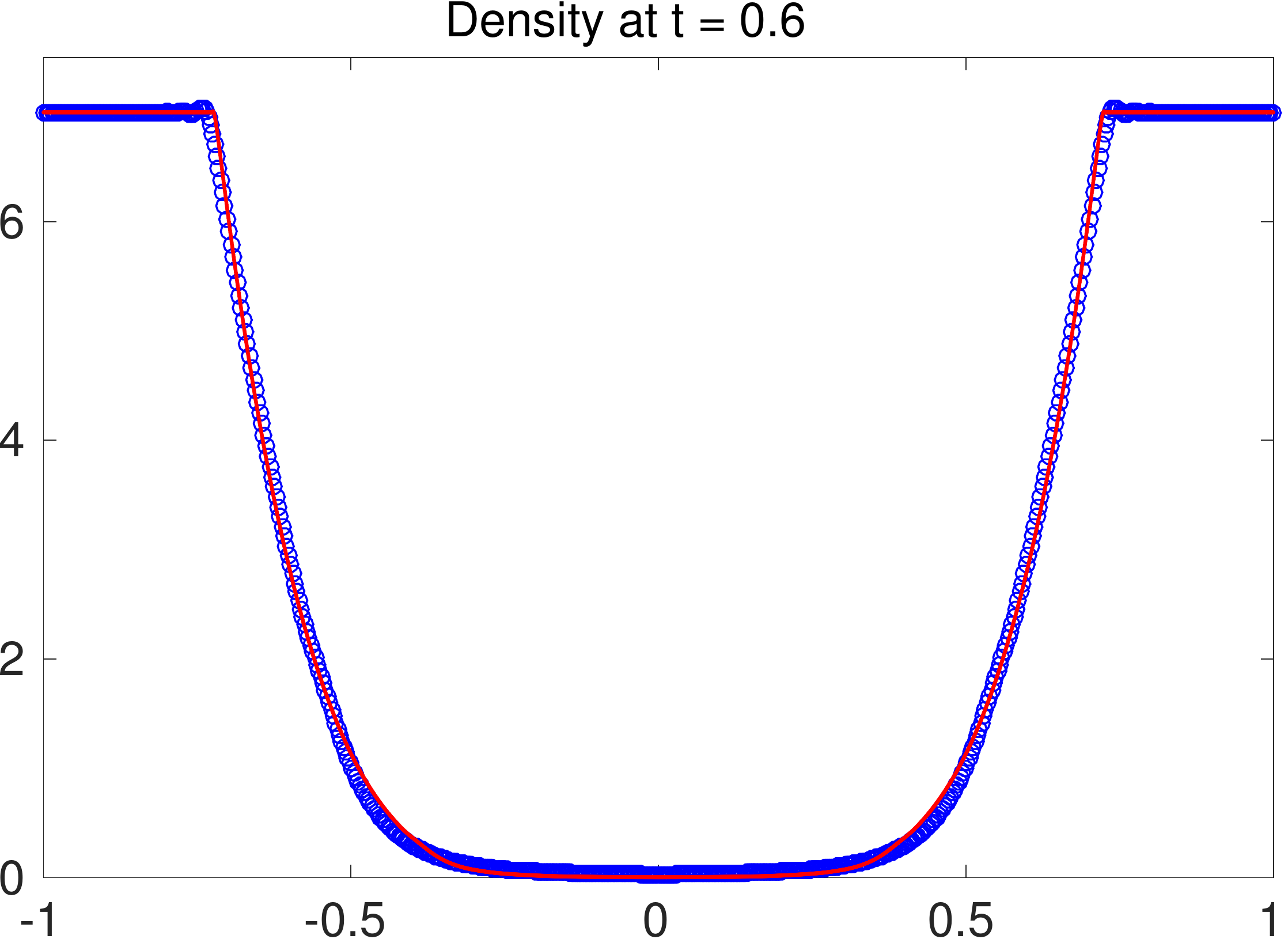}}
  \subfigure[]{\includegraphics[width=0.45 \linewidth]{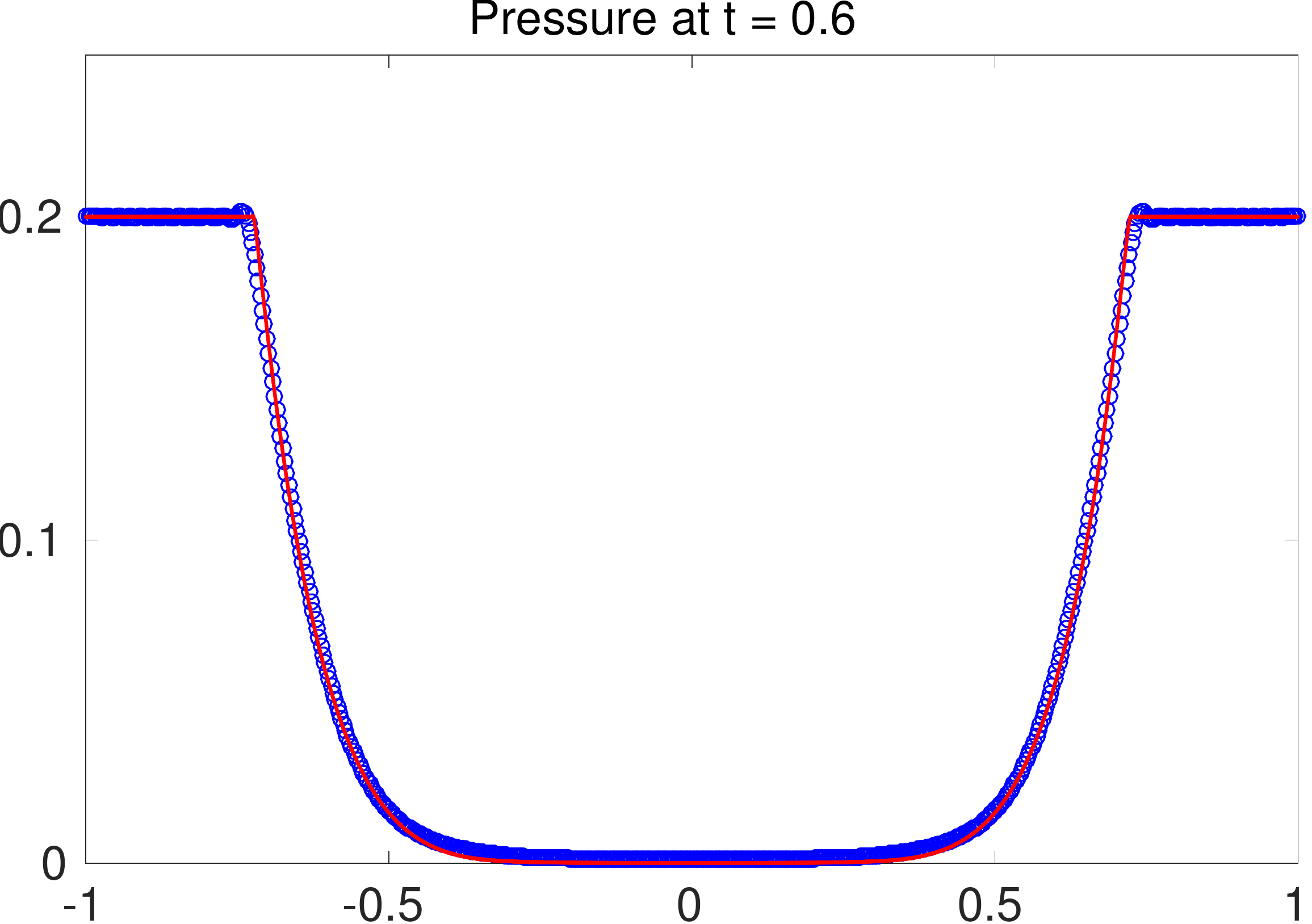}}
  \subfigure[]{\includegraphics[width=0.45 \linewidth]{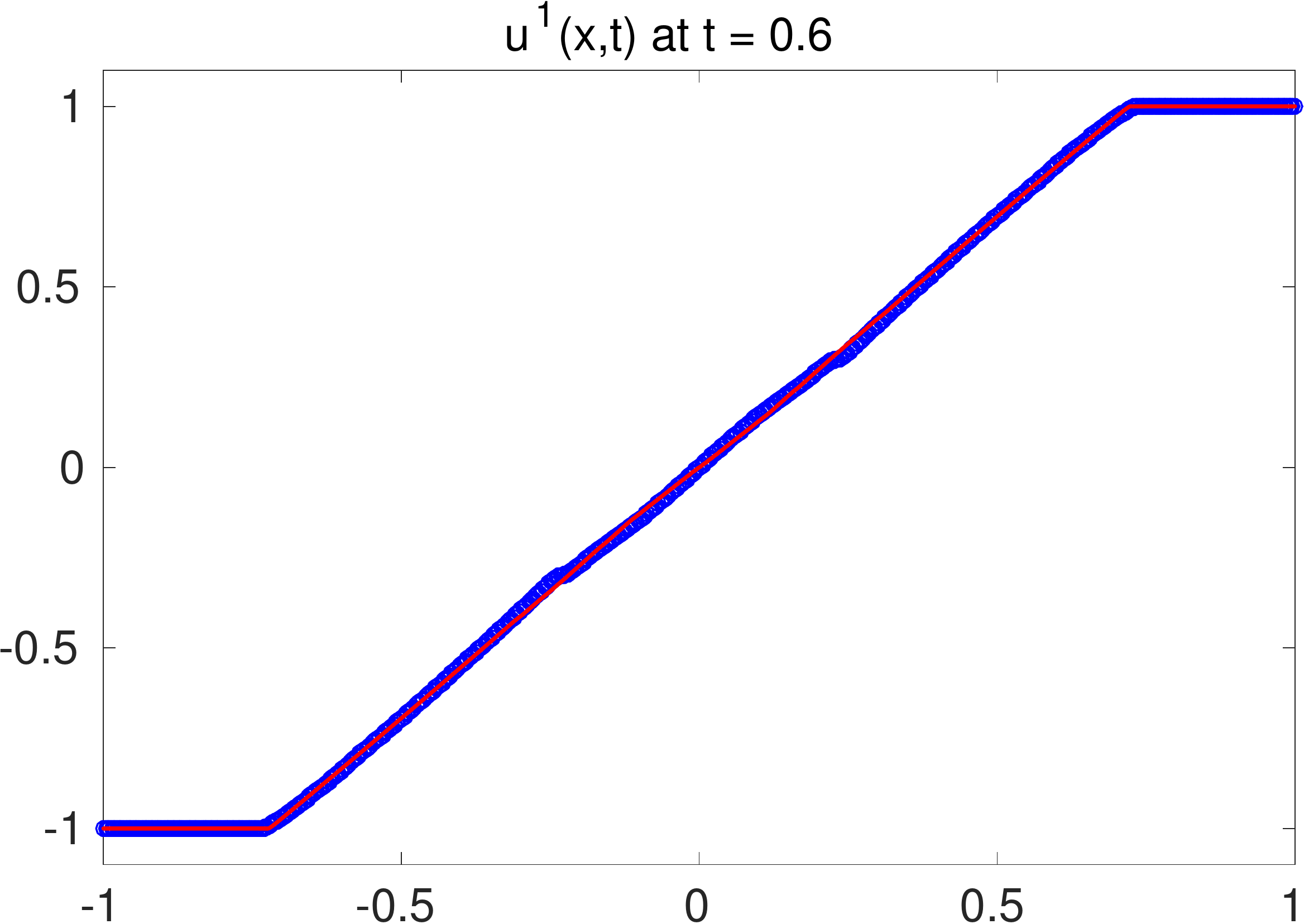}}
  \caption{The solution for the double-rarefaction problem. This is a standard example that fails for methods that are not positivity-preserving. The blue dots correspond to the computed numerical solution, and the red line corresponds to the computed
  solution on a highly refined mesh.
  \label{fig:1ddoublemach}
  }
 \end{center}
 \end{figure}

\subsection{\it Sedov blast wave}
\label{ex:Sedov1d}

This example is a simple one-dimensional model of an explosion that is difficult to
simulate without aggressive (or positivity-preserving) limiting. The initial
conditions involve one central cell with a large amount of
energy buildup that is surrounded by a large area of undisturbed air. These
initial conditions are supposed to approximate a delta function of energy.  As
time advances, a strong shock waves emanates from this central region and they move in 
opposite directions. This leaves the central post-shock regime with near zero density. 

The initial conditions are uniform in both density and velocity, with
$\rho=1$ and $u^1=0$.   The energy takes on the value $\E=\frac{3200000}{\Delta x}$ in the central cell and 
$\E=1.0\times 10^{-12}$ in every other cell.
This problem is explored extensively by Sedov, and in his classical text gives
an exact solution that we use to construct the exact solution underneath our
simulation \cite{sedov1993similarity}. 
We show our solution in Figure \ref{fig:1dsedov}, and we point out that our
results are quite good especially since we use such a coarse resolution of
size $\Delta x=\frac{1}{100}$. 

\begin{figure}
\begin{center}
 \subfigure[]{\includegraphics[width=0.45 \linewidth]{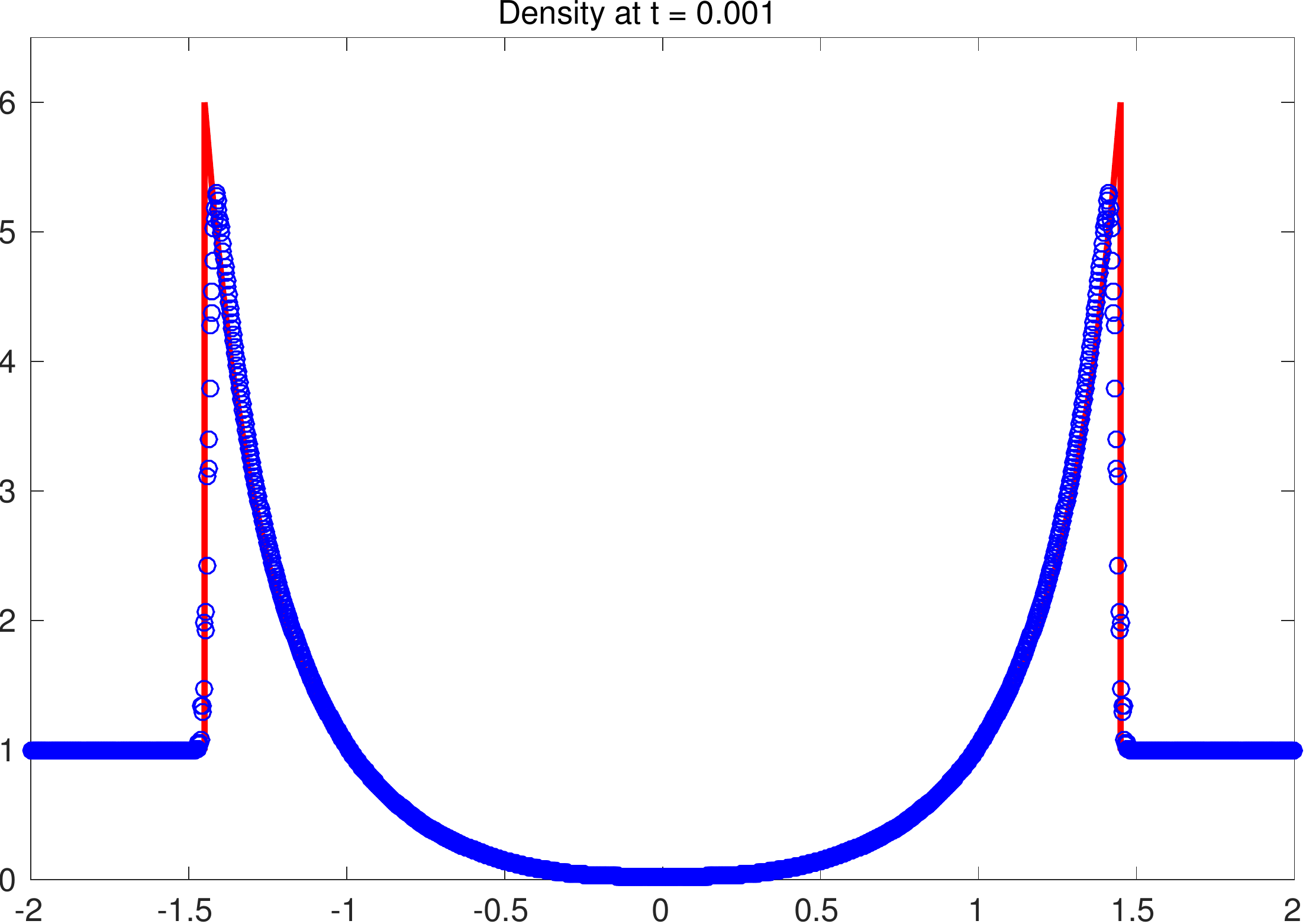}}
  \subfigure[]{\includegraphics[width=0.45 \linewidth]{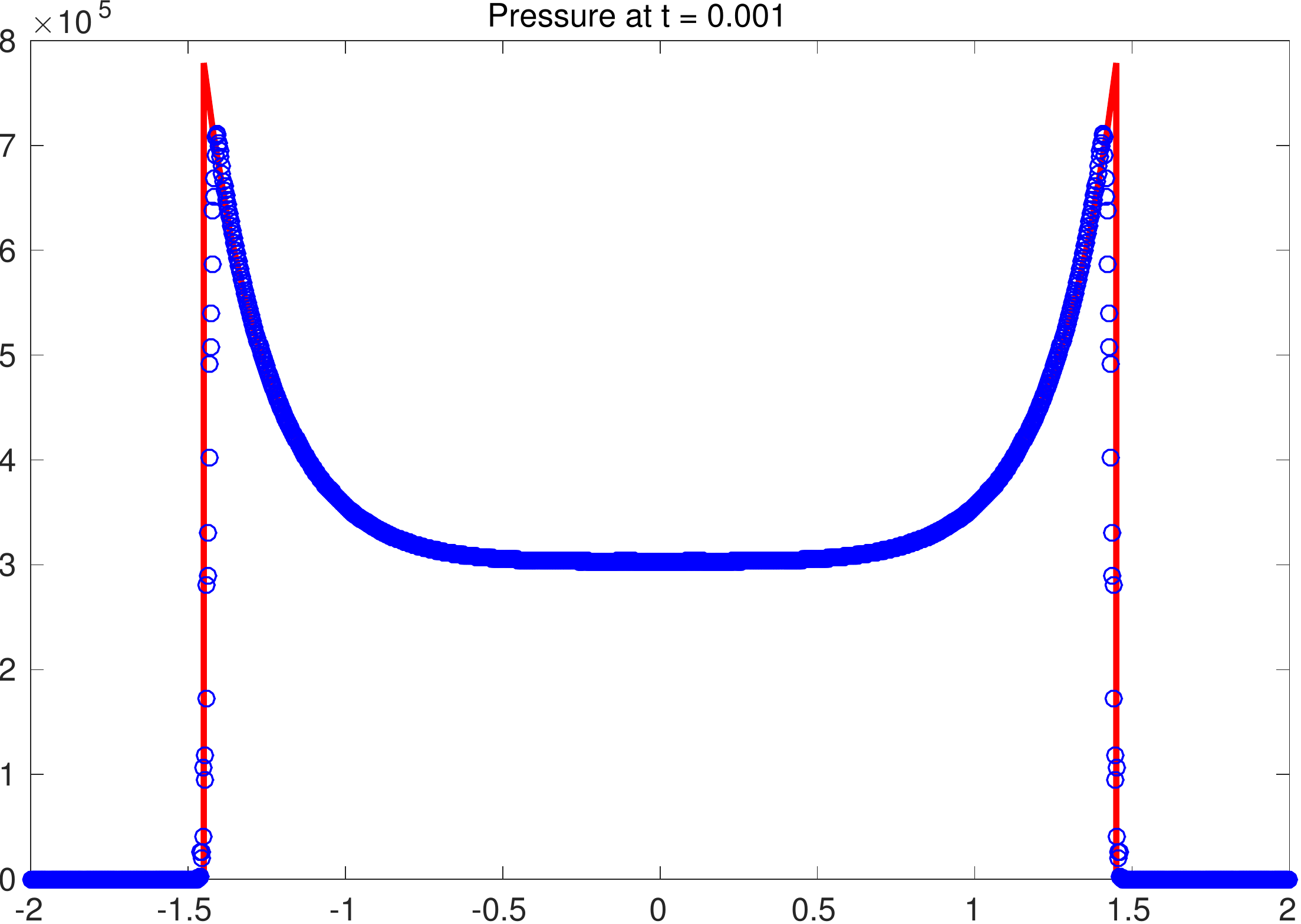}}
  \subfigure[]{\includegraphics[width=0.45 \linewidth]{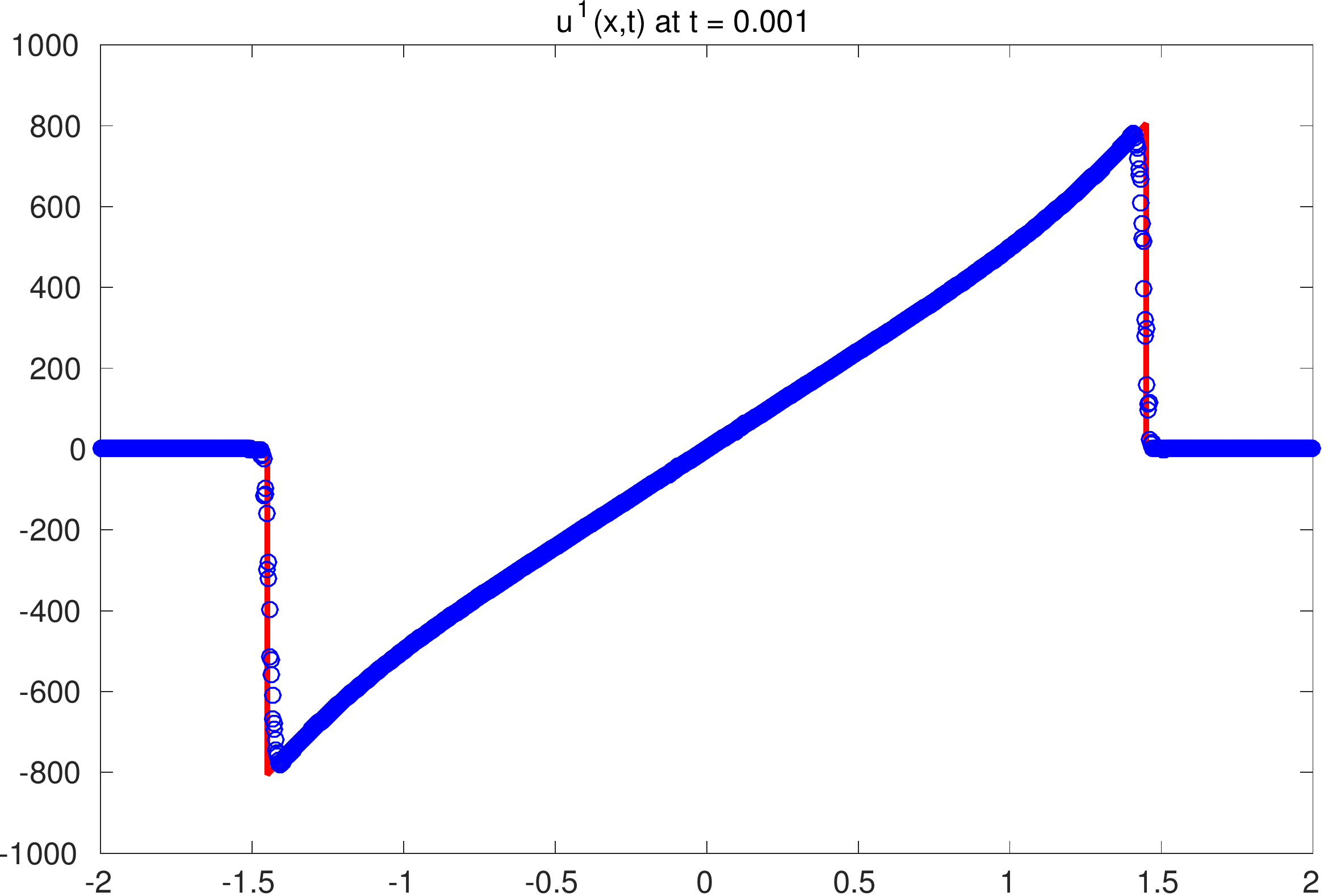}}
  \label{fig:1dsedov}
  \caption{Sedov Blast-Wave. This is another
  standard example that fails without the positivity limiting. The blue dots correspond to the computed numerical solution, and the red line corresponds to the exact
  solution.}
 \end{center}
 \end{figure}
 
\subsection{\it Two dimensional examples}

Here, we highlight the fact that this solver is able to operate on
both Cartesian and unstructured meshes.

\subsubsection{\it Convergence Results}
\label{subsubsec:ConvergenceTable}

We first verify the high-order accuracy of the proposed scheme.
For problems where the density and pressure are far away from zero, the
limiters proposed in this work ``turn off'', and therefore have no effect on
the solution.  In order to investigate the effect of this 
positivity preserving limiter, we simulate a
smooth problem where the solution has regions that are nearly zero. This is similar to
the smooth test case considered by \cite{article:ZhangShu10rectmesh}, \cite{seal2014explicit} and \cite{tang14}.
To this end, we consider initial conditions defined by
\begin{equation}
\begin{pmatrix}
   \rho_0 \\ u^1_0 \\ u^2_0 \\ p_0
  \end{pmatrix}({\bf x}) =
  \begin{pmatrix}
   1-0.9999\sin(2\pi x)\sin(2\pi y)\\ 1 \\ 0 \\ 1
  \end{pmatrix}
\end{equation}
on a computational domain of $[0,1]\times[0,1]$.
We integrate this problem up to a final time of $t=0.02$, and compute
$L^2$-norm errors against the exact solution given by 
\begin{equation}
 \begin{pmatrix}
       \rho \\ u^1 \\ u^2 \\ p
  \end{pmatrix}(t, {\bf x} ) =
  \begin{pmatrix}
       1-0.9999\sin(2\pi \left(x-t\right))\sin(2\pi y)\\ 1 \\ 0 \\ 1
  \end{pmatrix}.
  \label{eq:exactsolnconv}
\end{equation}

Results for 
Cartesian as well as unstructured meshes are presented in Table
\ref{table:convergence}.  These indicate that the high-order accuracy of the
method is not sacrificed when the limiters are turned on.
As a final note, we observe that it appears that in general much higher resolution is required on unstructured meshes
before the numerical results enter the asymptotic regime.

\begin{table}[ht]
\begin{center}
\hspace*{-50pt}\begin{tabular}{|r||c|c||c|c||c|c|}
\hline
\bf{\# Cartesian cells} & \bf{{Error}} & \bf{{Order}}& \bf{\# triangular cells} & \bf{{Error}}& \bf{{Order}}\\
\hline
\hline
$   784$ & $1.79\times 10^{-04}$ & --- & $20400$ & $1.03\times 10^{-05}$ & ---\\
\hline
$   1764$ & $7.08\times 10^{-06}$ & $7.966$ & $29280$ & $2.36\times 10^{-06}$ & $8.148$ \\
\hline
$  3969$ & $2.10\times 10^{-06}$ & $2.994$ & $42048$ & $2.43\times 10^{-07}$ & $12.560$ \\
\hline
$  8836$ & $6.22\times 10^{-07}$ & $3.045$ & $60550$ & $1.39\times 10^{-07}$ & $3.084$ \\
\hline
$  19881$ & $1.86\times 10^{-07}$ & $2.971$ & $86526$ & $8.11\times 10^{-08}$ & $2.999$ \\
\hline
$  44944$ & $5.48\times 10^{-08}$ & $3.000$ & $124998$ & $4.67\times 10^{-08}$ & $3.003$ \\
\hline
 --- & --- & --- & $179998$ & $2.70\times 10^{-08}$ & $3.013$ \\
\hline
\hline
\end{tabular}\hspace*{-50pt}

\caption{ 
    Convergence results for the 2D Euler problem in Section \ref{subsubsec:ConvergenceTable}. All errors are $L^2$ norm errors. We
    see that the positivity preserving limiter does not affect the asymptotic
    convergence rate of the method. Also note that this solution was run only to short time, as in \cite{article:ZhangShu10rectmesh}, \cite{seal2014explicit} and \cite{tang14}, because for this example we
    must resolve a positive, yet very small density leading to a very large wave-speed and thus a very small permissable time-step.} 
\label{table:convergence}

\end{center}
\end{table}

\subsubsection{\it Sedov blast on an unstructured mesh}
\label{ex:Sedov2d}

In this example we implement a two-dimensional version of the Sedov blast wave
on the circular domain 
\[
    \Omega = \left\{ (x,y) : x^2 + y^2 \leq 1.1 \right\}.
\]
The bulk of the domain begins with an undisturbed gas,
 $\vec{u} \equiv \vec{0}$, with uniform density $\rho \equiv 1$, and near-zero energy
$\E = 10^{-12}$.
Only the cells at the center of the domain contain a large amount of energy
that approximate a delta function.
To simulate this, we introduce a small region at the center of the domain that
is radially symmetric (because the simulation should be radially symmetric) of the
form
\[
    \E= \begin{cases}
        \frac{0.979264}{\pi r_d^2} & \sqrt{x^2+y^2}<r_d, \\
        10^{-12} & \text{otherwise},
     \end{cases}
\]
where $r_d = \sqrt{\frac{\pi 1.1^2}{\# \text{cells}}}$ is a characteristic length
of the mesh.


We present results in Figure \ref{fig:2dSedov} where we simulate our solution
with a total of $136270$ mesh cells. 


\begin{figure}[!ht]
\begin{center}
  \subfigure[]{\includegraphics[width=0.4 \linewidth]{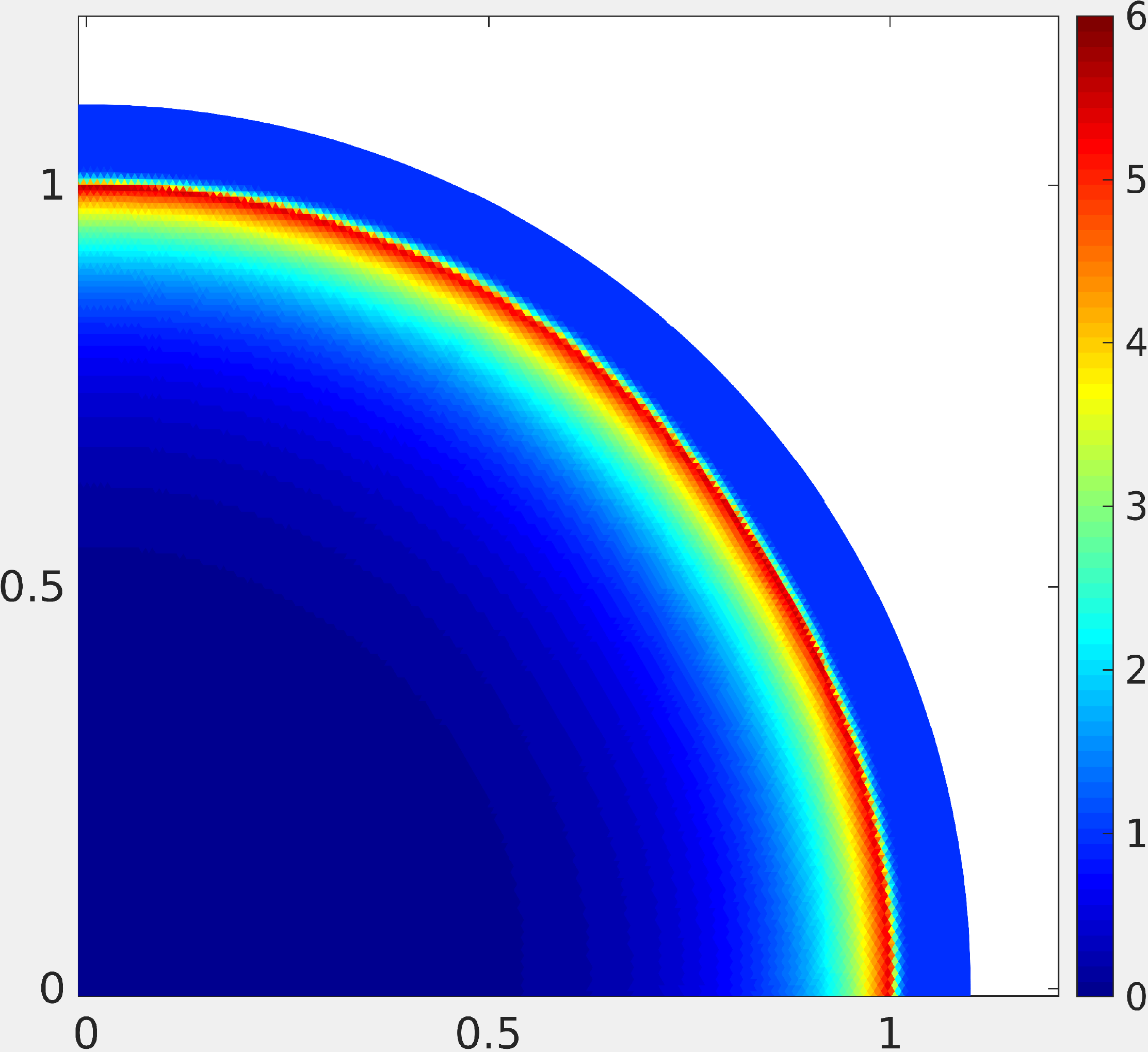}}\hspace{0.1 in}
  \subfigure[]{\includegraphics[width=0.53 \linewidth]{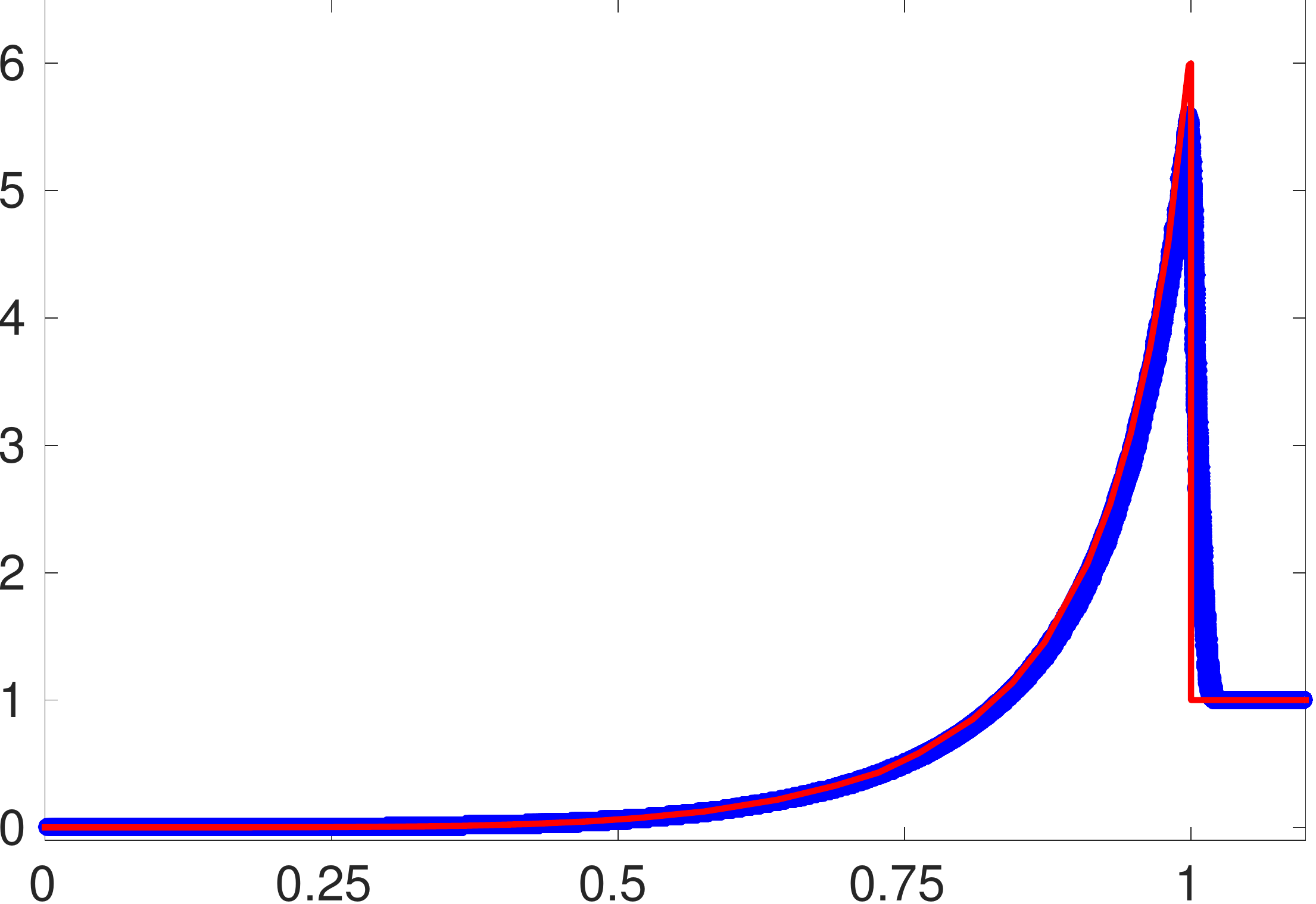}}
\caption{Sedov blast problem.  Here, we show density plots of the two
dimensional Sedov problem we introduce in \S\ref{ex:Sedov2d} .
The clear regions in the upper right part of subfigure (a) are due to the fact
that we only mesh the interior of the circular domain
$\left\{ (x,y) : x^2 + y^2 \leq 1.1 \right\}$.
Also note that we ran this problem on the entire
circular region but only plot the upper right region of the solution.
 \label{fig:2dSedov}
}
 \end{center}
 \end{figure}

\subsubsection{\it Shock-diffraction over a block step: Cartesian mesh}
\label{ex:ShockDiffStruct}

This is a common example used to test positivity limiters \cite{seal2014explicit,article:ZhangShu10rectmesh,ZhangShu11}. It involves a Mach
$5.09$ shock located above a step moving into air that is at rest with $\rho=1.4$ and $p=1.0$. The domain this problem is typically solved on
is $[0,1]\times [6,11] \cup [1,13]\times[0,11]$. The step is the region $[0,1] \times [0,6]$. Our boundary conditions are transparent everywhere
except above the step where they are inflow and on the surface of the step where we used solid wall. Our initial conditions have the shock located above
the step at $x=1$. The problem is typically run out to $t=2.3$, and if a positivity
limiter is not used the solution develops negative density and pressure
values, which causes the simulation to fail. The solution shown in Figure
\ref{fig:ShockDiffStruct} is run on a $390 \times 330$ Cartesian
mesh.


\begin{figure}[!ht]
\begin{center}
 \subfigure[density]{\includegraphics[width=0.48 \linewidth]{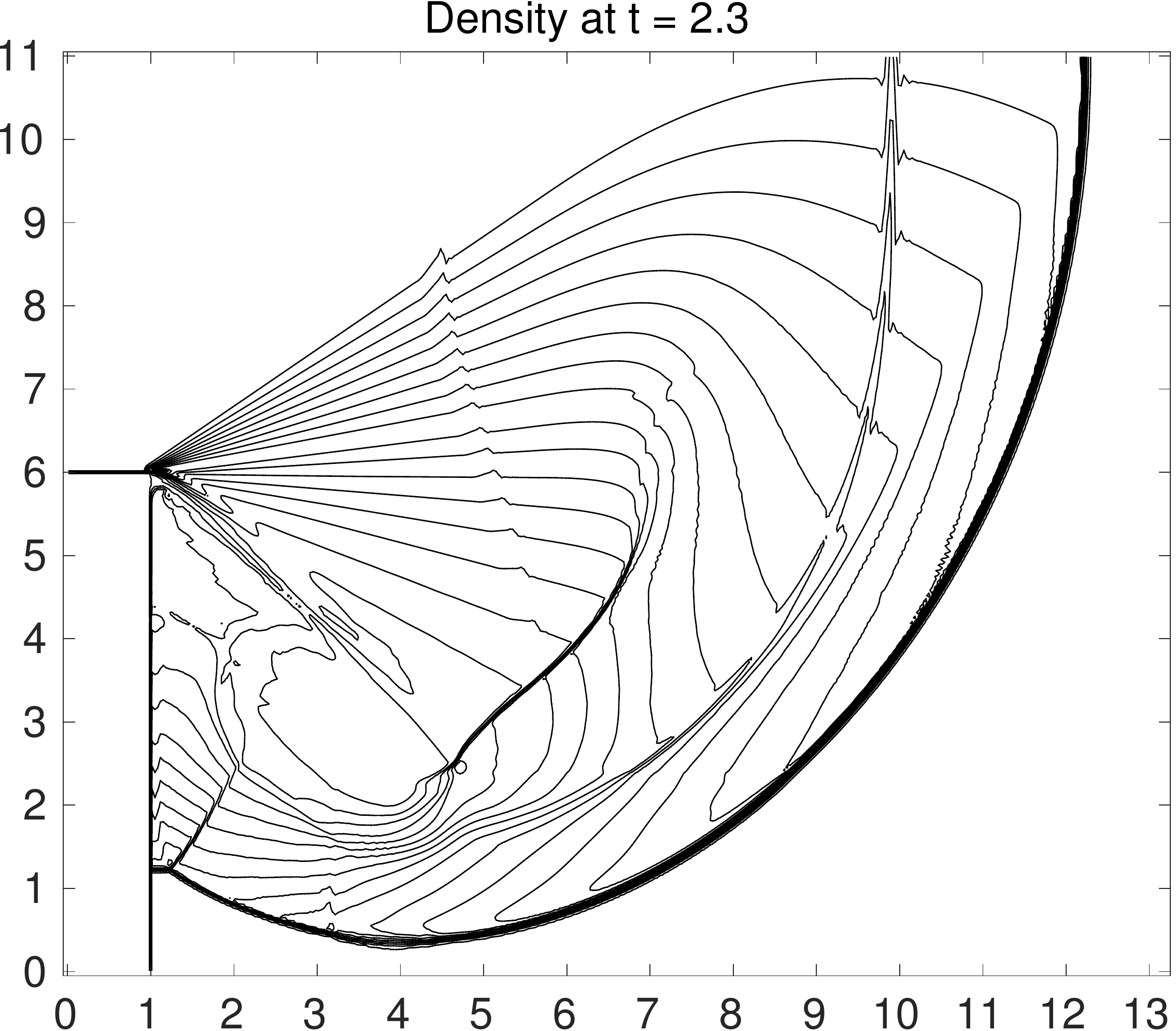}}
  \subfigure[]{\includegraphics[width=0.48 \linewidth]{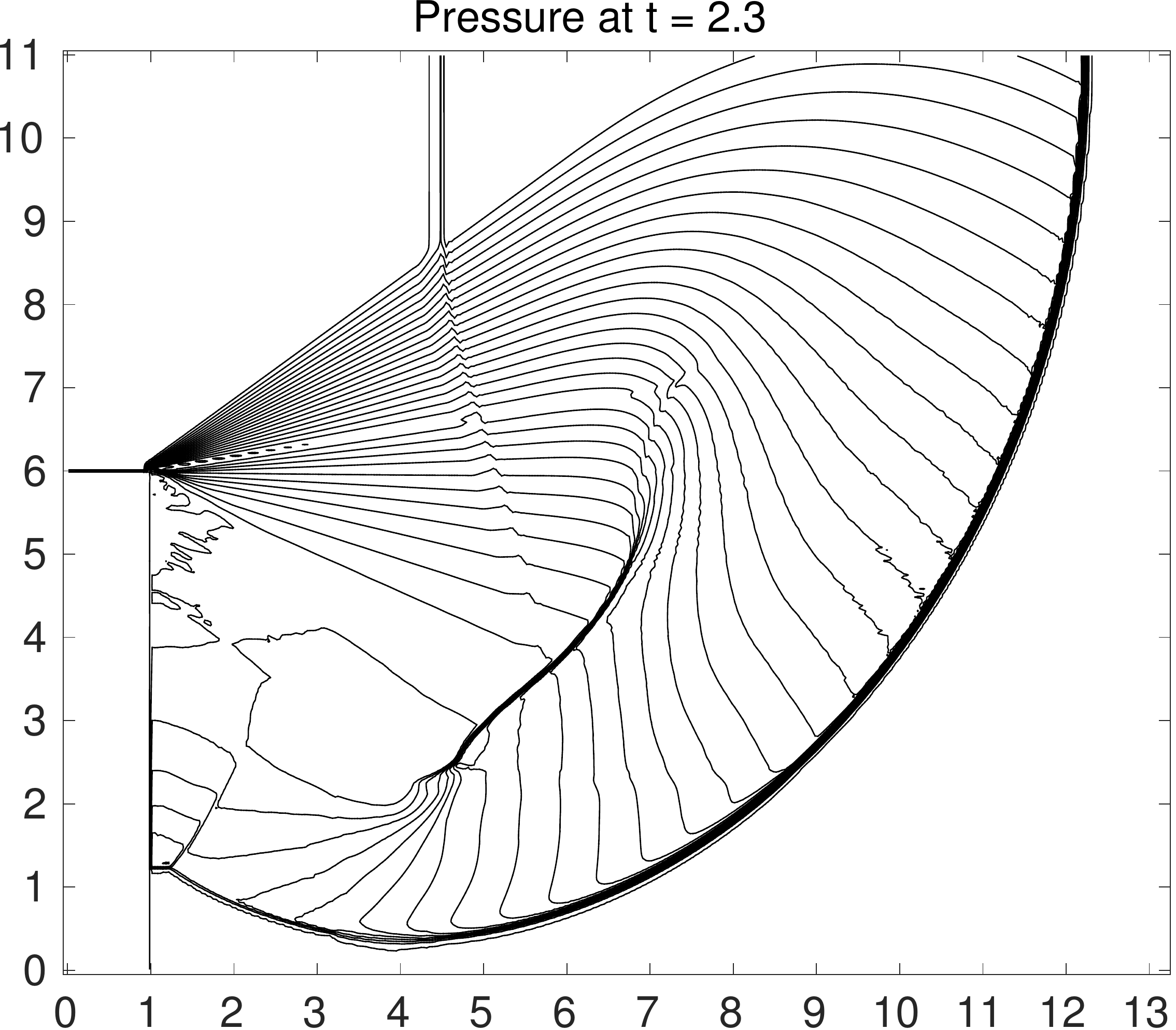}}
 \caption{The Mach $5.09$ shock-diffraction test problem on a
 $390\times 330$ Cartesian mesh.  
 For the density, we plot a total of $20$ equally spaced contour
 lines ranging from $\rho = 0.066227$ to $\rho = 7.0668$.  For the pressure, we plot a
 total of $40$ equally spaced contour
 lines ranging from $p = 0.091$ to $p = 37$ to match the figures in \cite{article:ZhangShu10rectmesh}.
 \label{fig:ShockDiffStruct}
 }
 \end{center}
\end{figure}

\subsubsection{\it Shock-diffraction over a block step: unstructured mesh}
\label{ex:ShockDiffUnstruct}

Next, we run the same problem from \S\ref{ex:ShockDiffStruct}, but we
discretize space using an unstructured triangular mesh with $126018$ cells. The results are
shown in Figure \ref{fig:ShockDiffUnstruct}, and indicate that the
unstructured solver behaves similarly to the Cartesian one.


\begin{figure}[!ht]
\begin{center}
 \subfigure[density]{\includegraphics[width=0.48 \linewidth]{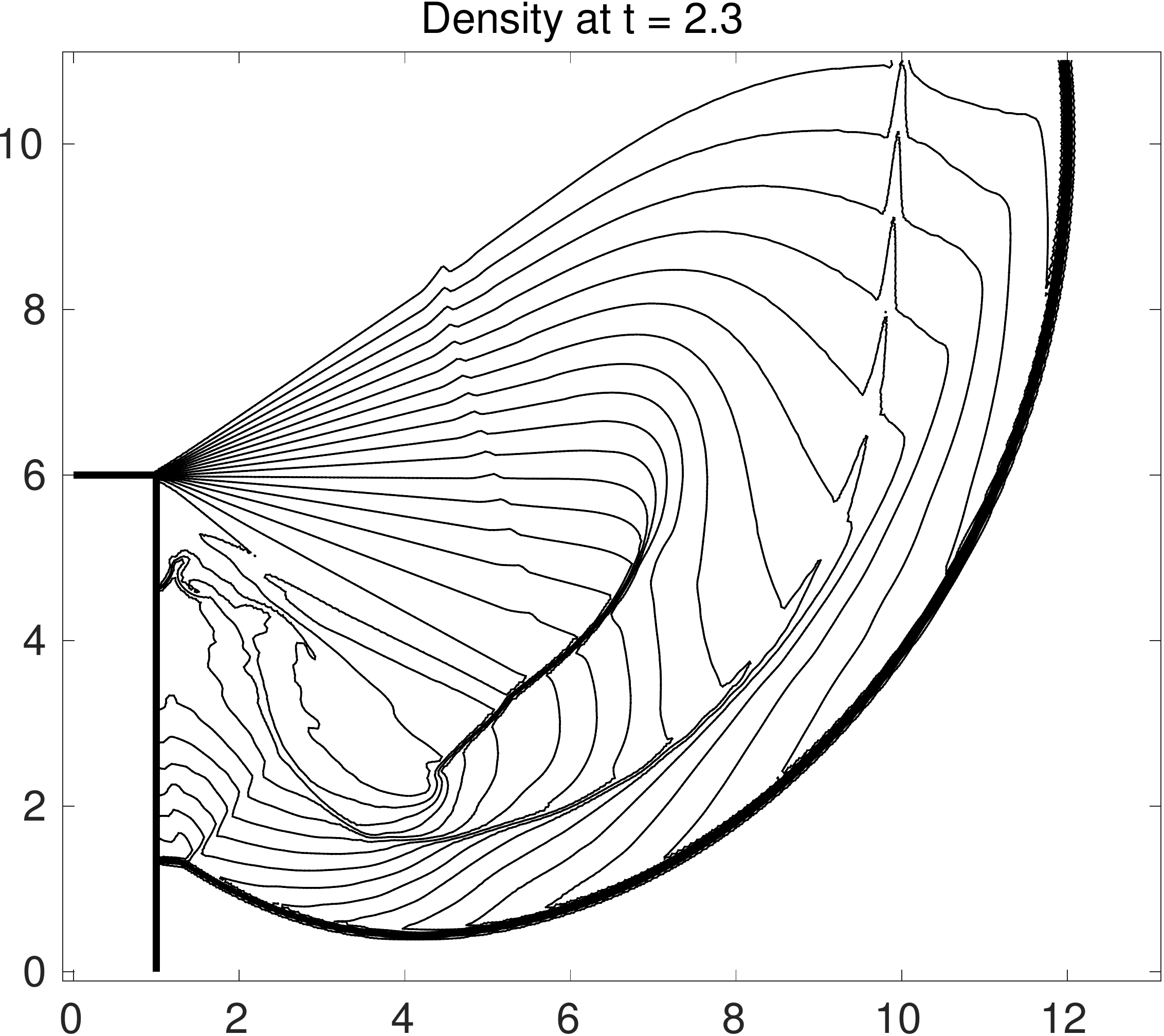}}
  \subfigure[]{\includegraphics[width=0.48 \linewidth]{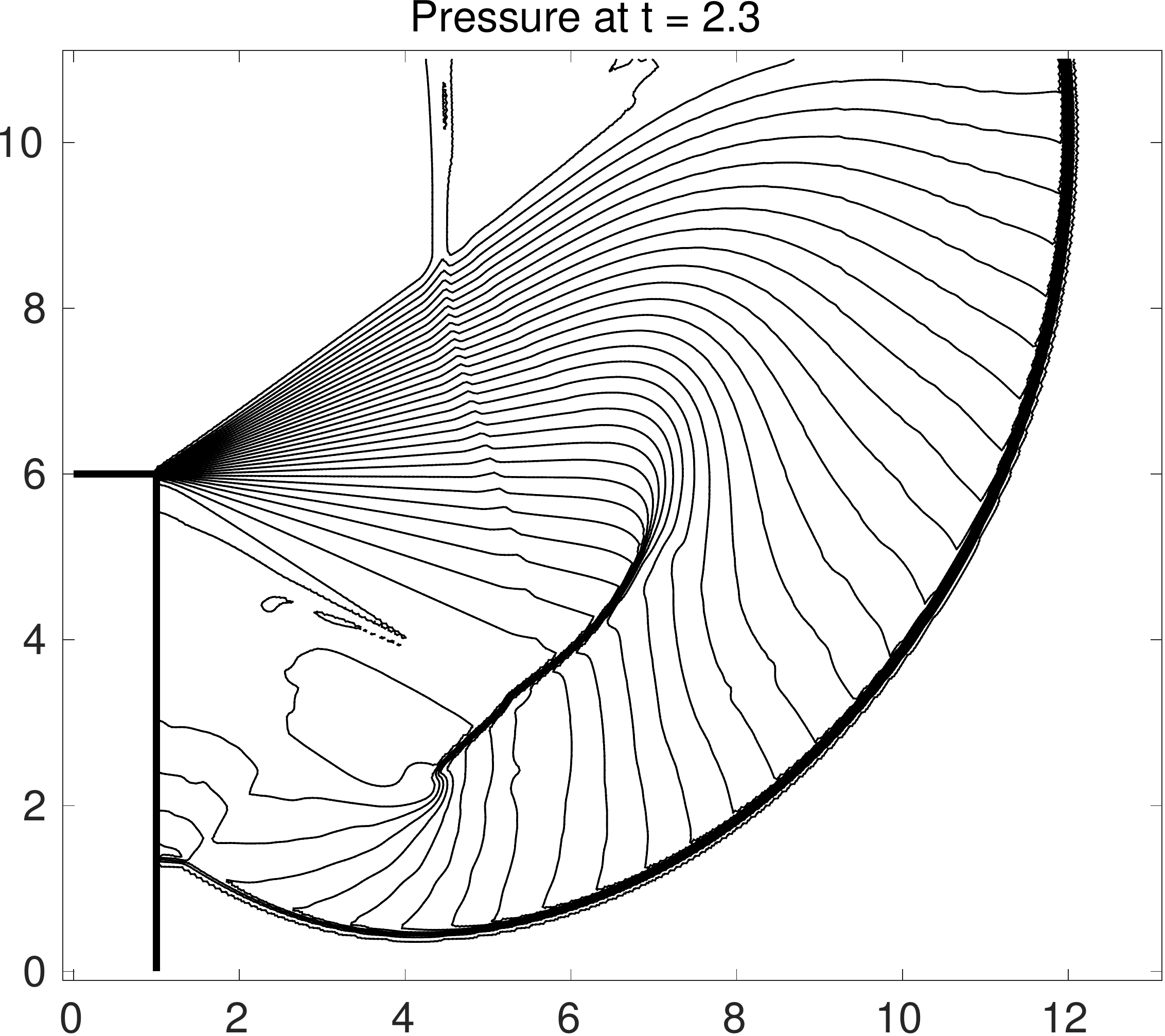}}
 \caption{The shock-diffraction test problem on an unstructured triangular mesh.
 For the density, we plot a total of $20$ equally spaced contour
 lines ranging from $\rho = 0.066227$ to $\rho = 7.0668$.  For the pressure, we plot a
 total of $40$ equally spaced contour
 lines ranging from $p = 0.091$ to $p = 37$, again to match the figures in \cite{article:ZhangShu10rectmesh}.
 \label{fig:ShockDiffUnstruct}
 }
 \end{center}
 \end{figure}

\subsubsection{\it Shock-diffraction over a $120$ degree wedge}
\label{ex:ShockDiffWedge}

This final shock-diffraction test problem is very similar to the previous test problems, however it must be run on an unstructured triangular
mesh because the wedge involved in this problem is triangular 
(with a $120$ degree angle)  \cite{ZhangShuTrimesh}.
Our domain for this problem is given by
\[
    [0,13]\times[0,11] \setminus [0,3.4] \cup [0,3.4] \times \left[\frac{6.0}{3.4}x, 6.0\right]. 
\]
Again the boundary conditions are transparent everywhere
except above the step where they are inflow and on the surface of the step
where they are reflective solid wall boundary conditions. In addition, the initial conditions for this problem
are also slightly different that those in the previous example.  Here, we have a Mach $10$ shock located above the step at $x=3.4$, and undisturbed air in the rest of the domain with
$\rho=1.4$ and $p=1.0$. This problem is run on an unstructured mesh with a
total of $122046$ cells. These results are presented in Figure \ref{fig:ShockDiffWedge}.

\begin{figure}[!ht]
\begin{center}
 \subfigure[]{\includegraphics[width=0.48 \linewidth]{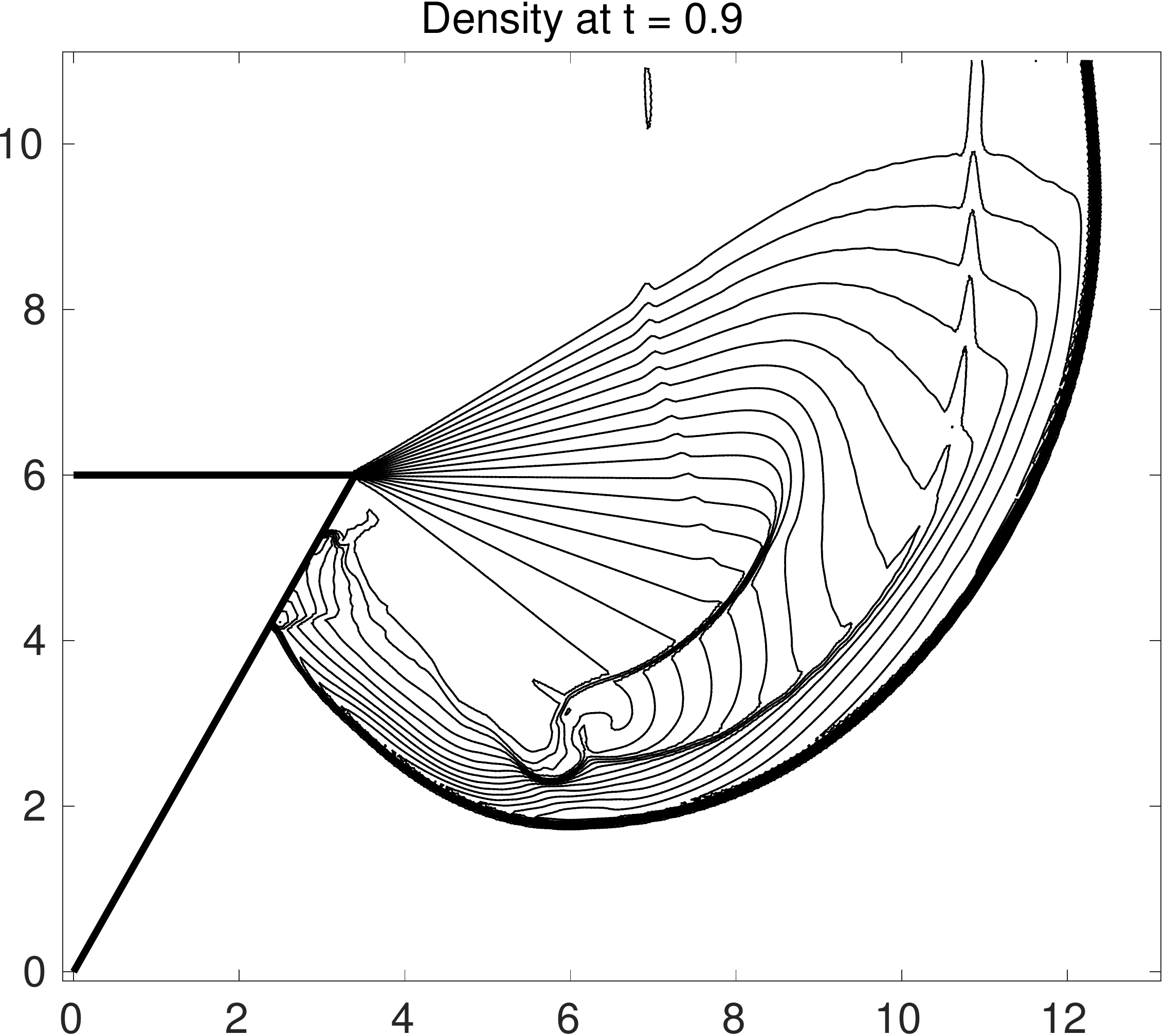}}
  \subfigure[]{\includegraphics[width=0.48 \linewidth]{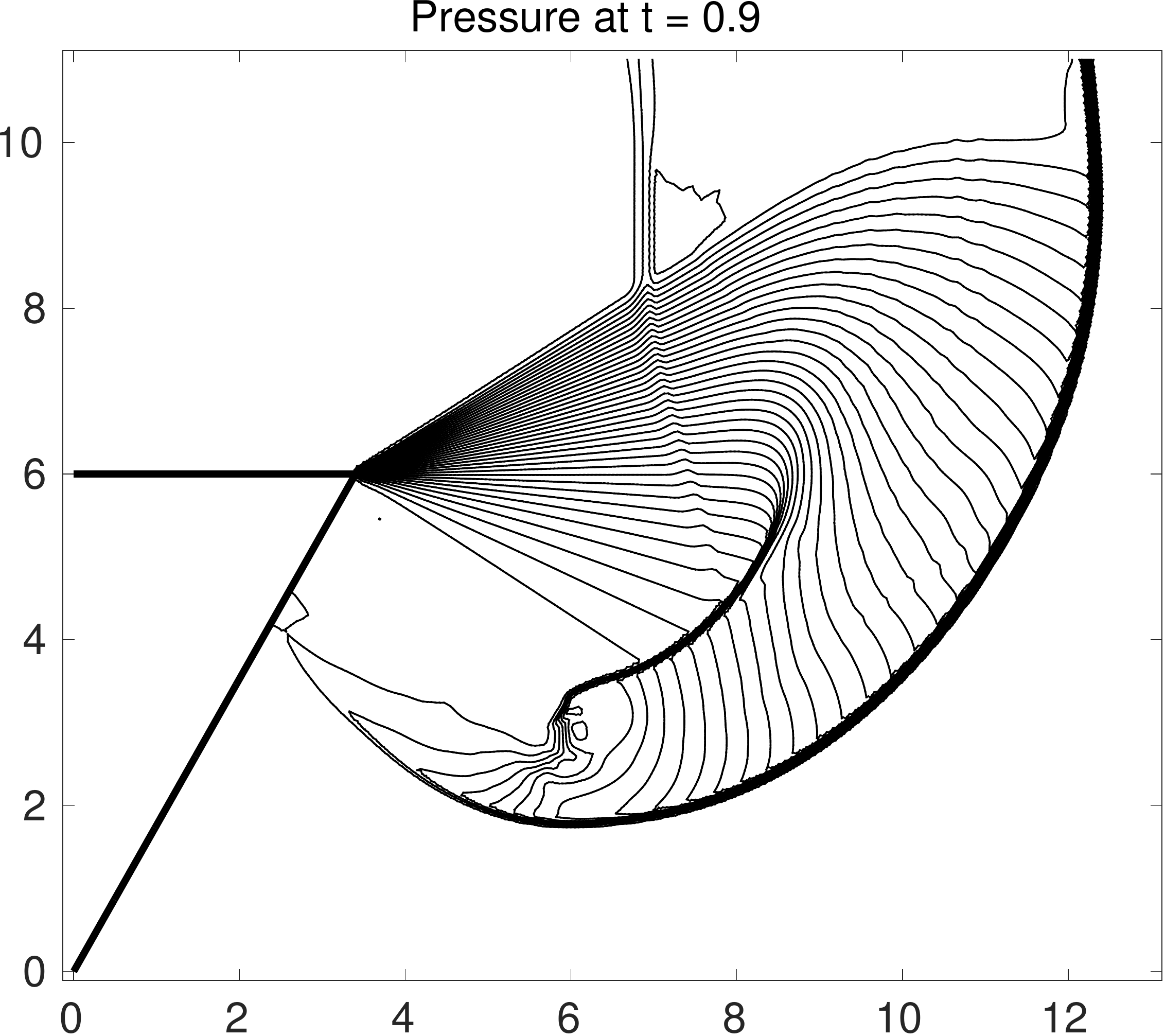}}
 \caption{Shock diffraction problem with a wedge.  Here, we present numerical results for the Mach $10$ shock diffraction problem, where the shock
 passes over a $120$ degree angular step.
 For the density, we plot a total of $20$ equally spaced contour
 lines ranging from $\rho = 0.0665$ to $\rho = 8.1$.  For the pressure, we plot a
 total of $40$ equally spaced contour
 lines ranging from $p = 0.5$ to $p = 118$.
 \label{fig:ShockDiffWedge}  
 }
\end{center}
\end{figure}
 

\section{Conclusions}
\label{sec:conclusions}

In this work we developed a novel positivity-preserving limiter for the Lax-Wendroff discontinuous Galerkin (LxW-DG) method.
Our results are high-order and applicable for unstructured meshes in multiple dimensions.  
Positivity of the solution is realized by leveraging two separate ideas: the
moment limiting work of Zhang and Shu \cite{ZhangShu11}, as well as the flux corrected
transport work of Xu and collaborators \cite{xu2013,liang2014parametrized,tang14,christlieb2015high,seal2014explicit}.  The additional shock
capturing limiter, which is required to obtain non-oscillatory results,
is the one recently developed by the current authors \cite{MoeRossSe15}.
Numerical results indicate the robustness of the method, and are promising for future applications to more complicated problems such as the ideal magnetohydrodynamics equations.
Future work includes introducing source terms to the solver, as well as pushing these methods to higher orders (e.g., $11^{\text{th}}$-order), but that requires either (a) an expedited way of computing
higher derivatives of the solution, or (b) rethinking how Runge-Kutta methods are applied in a modified flux framework (e.g., \cite{SeGuCh15}).

\bigskip

\noindent
{\bf Acknowledgements.}

\noindent
The work of SAM was supported in part by NSF grant DMS--1216732. 
The work of JAR was supported in part by NSF grant DMS--1419020.




\bibliographystyle{plain}                  

\begin{thebibliography}{10}

\bibitem{FCTPavel}
P.~Bochev, D.~Ridzal, G.~Scovazzi, and M.~Shashkov.
\newblock Formulation, analysis and numerical study of an optimization-based
  conservative interpolation (remap) of scalar fields for arbitrary
  {L}agrangian-{E}ulerian methods.
\newblock {\em J. Comput. Phys.}, 230(13):5199--5225, 2011.

\bibitem{book1981finite}
D.L. Book.
\newblock Finite-difference techniques for vectorized fluid dynamics
  calculations.
\newblock {\em New York and Berlin, Springer-Verlag, 1981. 233 p}, 1, 1981.

\bibitem{book1975flux}
D.L. Book, J.P. Boris, and K.~Hain.
\newblock Flux-corrected transport {II}: {G}eneralizations of the method.
\newblock {\em Journal of Computational Physics}, 18(3):248--283, 1975.

\bibitem{boris1973flux}
J.P. Boris and D.L. Book.
\newblock Flux-corrected transport. {I}. {SHASTA}, {A} fluid transport
  algorithm that works.
\newblock {\em Journal of computational physics}, 11(1):38--69, 1973.

\bibitem{boris1976flux}
J.P. Boris and D.L. Book.
\newblock Flux-corrected transport. {III}. {M}inimal-error {FCT} algorithms.
\newblock {\em Journal of Computational Physics}, 20(4):397--431, 1976.

\bibitem{ChFeSeTa2015}
A.J. Christlieb, X.~Feng, D.C. Seal, and Q.~Tang.
\newblock A high-order positivity-preserving single-stage single-step method
  for the ideal magnetohydrodynamic equations.
\newblock {\em arXiv preprint arXiv:1509.09208}, 2015.

\bibitem{SeGuCh15}
A.J. Christlieb, Y.~G{\"u}{\c{c}}l{\"u}, and D.C. Seal.
\newblock The {P}icard integral formulation of weighted essentially
  nonoscillatory schemes.
\newblock {\em SIAM J. Numer. Anal.}, 53(4):1833--1856, 2015.

\bibitem{christlieb2015high}
A.J. Christlieb, Y.~Liu, Q.~Tang, and Z.~Xu.
\newblock High order parametrized maximum-principle-preserving and
  positivity-preserving {WENO} schemes on unstructured meshes.
\newblock {\em J. Comput. Phys.}, 281:334--351, 2015.

\bibitem{tang14}
A.J. Christlieb, Y.~Liu, Q.~Tang, and Z.~Xu.
\newblock Positivity-preserving finite difference weighted {ENO} schemes with
  constrained transport for ideal magnetohydrodynamic equations.
\newblock {\em SIAM J. Sci. Comput.}, 37(4):A1825--A1845, 2015.

\bibitem{TVBRKDG4}
B.~Cockburn, S.~Hou, and C.-W. Shu.
\newblock The {R}unge-{K}utta local projection discontinuous {G}alerkin finite
  element method for conservation laws. {IV}. {T}he multidimensional case.
\newblock {\em Math. Comp.}, 54(190):545--581, 1990.

\bibitem{CoKarn}
B.~Cockburn, G.E. Karniadakis, and C.-W. Shu.
\newblock The development of discontinuous {G}alerkin methods.
\newblock In {\em Discontinuous {G}alerkin methods ({N}ewport, {RI}, 1999)},
  volume~11 of {\em Lect. Notes Comput. Sci. Eng.}, pages 3--50. Springer,
  Berlin, 2000.

\bibitem{TVBRKDG3}
B.~Cockburn, S.Y. Lin, and C.-W. Shu.
\newblock T{VB} {R}unge-{K}utta local projection discontinuous {G}alerkin
  finite element method for conservation laws. {III}. {O}ne-dimensional
  systems.
\newblock {\em J. Comput. Phys.}, 84(1):90--113, 1989.

\bibitem{TVBRKDG2}
B.~Cockburn and C.-W. Shu.
\newblock T{VB} {R}unge-{K}utta local projection discontinuous {G}alerkin
  finite element method for conservation laws. {II}. {G}eneral framework.
\newblock {\em Math. Comp.}, 52(186):411--435, 1989.

\bibitem{TVBRKDG5}
B.~Cockburn and C.-W. Shu.
\newblock The {R}unge-{K}utta discontinuous {G}alerkin method for conservation
  laws. {V}. {M}ultidimensional systems.
\newblock {\em J. Comput. Phys.}, 141(2):199--224, 1998.

\bibitem{CourantDifferences}
R.~Courant, E.~Isaacson, and M.~Rees.
\newblock On the solution of nonlinear hyperbolic differential equations by
  finite differences.
\newblock {\em Comm. Pure. Appl. Math.}, 5:243--255, 1952.

\bibitem{DuBaDiToMuDi08}
M.~Dumbser, D.S. Balsara, E.F. Toro, and C.-D. Munz.
\newblock A unified framework for the construction of one-step finite volume
  and discontinuous {G}alerkin schemes on unstructured meshes.
\newblock {\em J. Comput. Phys.}, 227(18):8209--8253, 2008.

\bibitem{dumbser2007arbitrary}
M.~Dumbser, M.~K\"aser, and E.F. Toro.
\newblock An arbitrary high-order discontinuous galerkin method for elastic
  waves on unstructured meshes-v. local time stepping and p-adaptivity.
\newblock {\em Geophysical Journal International}, 171(2):695--717, 2007.

\bibitem{dumbser2005ader}
M.~Dumbser and C.-D. Munz.
\newblock {ADER} discontinuous {G}alerkin schemes for aeroacoustics.
\newblock {\em Comptes Rendus M{\'e}canique}, 333(9):683--687, 2005.

\bibitem{DuMu06}
M.~Dumbser and C.-D. Munz.
\newblock Building blocks for arbitrary high order discontinuous {G}alerkin
  schemes.
\newblock {\em J. Sci. Comput.}, 27(1-3):215--230, 2006.

\bibitem{DuZaHiBa13}
M.~Dumbser, O.~Zanotti, A.~Hidalgo, and D.S. Balsara.
\newblock A{DER}-{WENO} finite volume schemes with space-time adaptive mesh
  refinement.
\newblock {\em J. Comput. Phys.}, 248:257--286, 2013.

\bibitem{Gassner11}
G.~Gassner, M.~Dumbser, F.~Hindenlang, and C.-D. Munz.
\newblock Explicit one-step time discretizations for discontinuous {G}alerkin
  and finite volume schemes based on local predictors.
\newblock {\em J. Comput. Phys.}, 230(11):4232--4247, 2011.

\bibitem{Godunov}
S.K. Godunov.
\newblock Difference method of computation of shock waves.
\newblock {\em Uspehi Mat. Nauk (N.S.)}, 12(1(73)):176--177, 1957.

\bibitem{GoShuTa01}
S.~Gottlieb, C.-W. Shu, and E.~Tadmor.
\newblock Strong stability-preserving high-order time discretization methods.
\newblock {\em SIAM Rev.}, 43(1):89--112 (electronic), 2001.

\bibitem{GuoQiuQiu14}
W.~Guo, J.-M. Qiu, and J.~Qiu.
\newblock A new {L}ax--{W}endroff discontinuous {G}alerkin method with
  superconvergence.
\newblock {\em J. Sci. Comput.}, 65(1):299--326, 2015.

\bibitem{harten1972self}
A.~Harten and G.~Zwas.
\newblock Self-adjusting hybrid schemes for shock computations.
\newblock {\em Journal of Computational Physics}, 9(3):568--583, 1972.

\bibitem{Kraa1991}
J.F.B.M. Kraaijevanger.
\newblock Contractivity of {R}unge-{K}utta methods.
\newblock {\em BIT}, 31(3):482--528, 1991.

\bibitem{kuzmin05}
D.~Kuzmin and R.~L\"ohner, editors.
\newblock {\em Flux-corrected transport}.
\newblock Scientific Computation. Springer-Verlag, Berlin, 2005.
\newblock Principles, algorithms, and applications.

\bibitem{LaxWendroff}
P.~Lax and B.~Wendroff.
\newblock Systems of conservation laws.
\newblock {\em Comm. Pure Appl. Math.}, 13:217--237, 1960.

\bibitem{Lax}
P.D. Lax.
\newblock Weak solutions of nonlinear hyperbolic equations and their numerical
  computation.
\newblock {\em Comm. Pure Appl. Math.}, 7:159--193, 1954.

\bibitem{liang2014parametrized}
C.~Liang and Z.~Xu.
\newblock Parametrized maximum principle preserving flux limiters for high
  order schemes solving multi-dimensional scalar hyperbolic conservation laws.
\newblock {\em J. Sci. Comput.}, 58(1):41--60, 2014.

\bibitem{MoeRossSe15}
S.A. Moe, J.A. Rossmanith, and D.C. Seal.
\newblock A simple and effective high-order shock-capturing limiter for
  discontinuous {G}alerkin methods.
\newblock {\em arXiv preprint arXiv:1507.03024v1}, 2015.

\bibitem{VNeumann}
J.~Von Neumann and R.D. Richtmyer.
\newblock A method for the numerical calculation of hydrodynamic shocks.
\newblock {\em J. Appl. Phys.}, 21:232--237, 1950.

\bibitem{PerthameShu}
B.~Perthame and C.-W. Shu.
\newblock On positivity preserving finite volume schemes for {E}uler equations.
\newblock {\em Numer. Math.}, 73(1):119--130, 1996.

\bibitem{QiuDumbserShu05}
J.~Qiu, M.~Dumbser, and C.-W. Shu.
\newblock The discontinuous {G}alerkin method with {L}ax-{W}endroff type time
  discretizations.
\newblock {\em Comput. Methods Appl. Mech. Eng.}, 194(42-44):4528--4543, 2005.

\bibitem{dogpack}
J.A. Rossmanith.
\newblock {\sc DoGPack} software, 2015.
\newblock Available from {\tt http://www.dogpack-code.org}.

\bibitem{RuuthSpiteri2002}
S.J. Ruuth and R.J. Spiteri.
\newblock Two barriers on strong-stability-preserving time discretization
  methods.
\newblock In {\em Proceedings of the {F}ifth {I}nternational {C}onference on
  {S}pectral and {H}igh {O}rder {M}ethods ({ICOSAHOM}-01) ({U}ppsala)},
  volume~17, pages 211--220, 2002.

\bibitem{FINESS}
D.C. Seal.
\newblock {\sc FINESS} software, 2015.
\newblock Available from \\ {\tt https://bitbucket.org/dseal/finess}.

\bibitem{SeGuCh14}
D.C. Seal, Y.~G{\"u}{\c{c}}l{\"u}, and A.J. Christlieb.
\newblock High-order multiderivative time integrators for hyperbolic
  conservation laws.
\newblock {\em J. Sci. Comput.}, 60(1):101--140, 2014.

\bibitem{seal2014explicit}
D.C. Seal, Q.~Tang, Z.~Xu, and A.J. Christlieb.
\newblock An explicit high-order single-stage single-step positivity-preserving
  finite difference {WENO} method for the compressible {E}uler equations.
\newblock {\em Journal of Scientific Computing}, pages 1--20, 2015.

\bibitem{sedov1993similarity}
L.I. Sedov.
\newblock {\em Similarity and dimensional methods in mechanics}.
\newblock Academic Press, New York-London, 1959.

\bibitem{SOD}
G.A. Sod.
\newblock A survey of several finite difference methods for systems of
  nonlinear hyperbolic conservation laws.
\newblock {\em J. Computational Phys.}, 27(1):1--31, 1978.

\bibitem{TaDuBaDiMu07}
A.~Taube, M.~Dumbser, D.S. Balsara, and C.-D. Munz.
\newblock Arbitrary high-order discontinuous {G}alerkin schemes for the
  magnetohydrodynamic equations.
\newblock {\em J. Sci. Comput.}, 30(3):441--464, 2007.

\bibitem{proceedings:TitarevToro02}
V.A. Titarev and E.F. Toro.
\newblock A{DER}: arbitrary high order {G}odunov approach.
\newblock In {\em Proceedings of the {F}ifth {I}nternational {C}onference on
  {S}pectral and {H}igh {O}rder {M}ethods ({ICOSAHOM}-01) ({U}ppsala)},
  volume~17, pages 609--618, 2002.

\bibitem{ullrich2014flux}
P.A. Ullrich and M.R. Norman.
\newblock The flux-form semi-{L}agrangian spectral element ({FF}--{SLSE})
  method for tracer transport.
\newblock {\em Quarterly Journal of the Royal Meteorological Society},
  140(680):1069--1085, 2014.

\bibitem{xu2013}
Z.~Xu.
\newblock Parametrized maximum principle preserving flux limiters for high
  order schemes solving hyperbolic conservation laws: {O}ne-dimensional scalar
  problem.
\newblock {\em Math. Comp.}, 83(289):2213--2238, 2014.

\bibitem{ZalesakStruct}
S.T. Zalesak.
\newblock The design of flux-corrected transport ({FCT}) algorithms for
  structured grids.
\newblock In {\em Flux-corrected transport}, Sci. Comput., pages 29--78.
  Springer, Berlin, 2005.

\bibitem{article:ZhangShu10rectmesh}
X.~Zhang and C.-W. Shu.
\newblock {On positivity preserving high order discontinuous Galerkin schemes
  for compressible Euler equations on rectangular meshes}.
\newblock {\em J. Comp. Phys.}, 229:8918---8934, 2010.

\bibitem{ZhangShu11}
X.~Zhang and C.-W. Shu.
\newblock Maximum-principle-satisfying and positivity-preserving high-order
  schemes for conservation laws: survey and new developments.
\newblock {\em Proc. R. Soc. A}, 467(2134):2752--2776, 2011.

\bibitem{ZhangShuTrimesh}
X.~Zhang, Y.~Xia, and C.-W. Shu.
\newblock Maximum-principle-satisfying and positivity-preserving high order
  discontinuous {G}alerkin schemes for conservation laws on triangular meshes.
\newblock {\em J. Sci. Comput.}, 50(1):29--62, 2012.

\bibitem{zheng2006non}
H.~Zheng, Z.~Zhang, and E.~Liu.
\newblock Non-linear seismic wave propagation in anisotropic media using the
  flux-corrected transport technique.
\newblock {\em Geophysical Journal International}, 165(3):943--956, 2006.

\end{thebibliography}



\end{document}